\documentclass[11pt]{amsart}
\usepackage{latexsym,amsfonts,amssymb,graphicx,amsmath}
\usepackage{amssymb}
\usepackage{enumerate}
\usepackage{enumitem}
\usepackage{booktabs}
\usepackage{stmaryrd}
\usepackage{sidecap}
\usepackage{color}

\numberwithin{equation}{section}

\newcommand{\n}[1]{\boldsymbol{#1}}
\newcommand{\C}{ \Upsilon}
\newcommand{\TG}{ T_\Gamma}
\newcommand{\Eh}{\mathcal{E}^h}
\newcommand{\Th}{\mathcal{T}^h}
\newcommand{\EhG}{\mathcal{E}_h^\Gamma}
\newcommand{\ThG}{\mathcal{T}_h^\Gamma}
\newcommand{\EhGn}{\mathcal{E}_h \backslash \mathcal{E}_h^\Gamma}
\newcommand{\wT}{\omega_T}
\newcommand{\wTG}{\omega_{T}^{\Gamma}}

\newcommand{\revblue}[1]{\textcolor{black}{{#1}}}

\newcommand{\ra}[1]{\renewcommand{\arraystretch}{#1}}

\newcommand{\Ce}{C_E}
\newcommand{\Cr}{C_{\text{reg}}}

\usepackage{chngcntr}

\theoremstyle{plain}
\newtheorem{theorem}{\bf Theorem}
\newtheorem{lemma}{Lemma}
\newtheorem{proposition}{Proposition}
\newtheorem{corollary}{Corollary}
\newtheorem*{Assmptn}{Assumption}

\theoremstyle{remark}

\theoremstyle{definition}
\newtheorem{Def}{Definition}[section]

\begin{document}
\title[FEM for interface]{A finite element method for  high-contrast interface problems with error estimates independent of contrast}

\author{Johnny Guzm\'an\textsuperscript{1}}
\thanks{\textsuperscript{1} Supported by NSF- DMS 1318108}
\address{\textsuperscript{1} Division of Applied Mathematics, Brown University, Providence, RI 02912, USA}
\email{johnny\_guzman@brown.edu}
\author{Manuel A. S\'anchez\textsuperscript{2}}
\address{\textsuperscript{2} School of Mathematics, University of Minnesota, Minneapolis, Minnesota 55455, USA}
\email{sanchez@umn.edu}
\author{Marcus Sarkis\textsuperscript{3}}
\thanks{\textsuperscript{3} Supported in part by NSF-MRI 1337943 and NSF-MPS 1522663}
\address{\textsuperscript{3}Department of Mathematical Sciences at Worcester Polytechnic Institute, 100 Institute Road, Worcester, MA 01609, USA }
\email{msarkis@wpi.edu}

\keywords{Interface problems, high-contrast, finite elements.}

\subjclass[2000]{65N30, 65N15.}
\date{}

\begin{abstract}
We define a new finite element method for a steady state  elliptic problem with discontinuous diffusion coefficients where the meshes are not aligned with the interface. We prove optimal error estimates in the $L^2$ norm and $H^1$ weighted semi-norm independent of the contrast between the coefficients. Numerical experiments validating our theoretical findings are provided.
\noindent
\end{abstract}

\maketitle

\section{Introduction}\label{section1}

In this article we develop a finite element method for a steady state interface problem. We pay particular attention to high-contrast problems, proving optimal error estimates independent of the contrast of the discontinuous constant coefficients for the numerical method.

Let $\Omega\subset \mathbb{R}^2$ be a polygonal domain with an  immersed interface $\Gamma$ such that $\overline{\Omega}=\overline{\Omega}^{-} \cup \,\overline{\Omega}^{+}$ with $\Omega^-\cap \Omega^+=\emptyset$, and $\Gamma=\overline{\Omega}^{-} \cap \,\overline{\Omega}^{+}$. We assume that $\Gamma$ does not intersect $\partial \Omega$, enclosing either $\Omega^-$ or $\Omega^+$.  Our numerical method will approximate a solution of the problem below.
\begin{subequations}\label{Problem}
\begin{alignat}{3}
    -\rho^{\pm} \Delta u^{\pm} &= f^{\pm} \qquad &\mbox{in } &\Omega^{\pm}, \label{Problem:a} \\
                                           u &=0            &\mbox{on }&\partial \Omega, \label{Problem:b}\\
                            \left[u\right]&=0            &\mbox{on }&\Gamma, \label{Problem:c} \\
  \left[\rho D_{\n{n}}  u  \right]&=0            &\mbox{on }&\Gamma. \label{Problem:d}
\end{alignat}
\end{subequations}
The jumps across the interface $\Gamma$ are defined as
\begin{align*}
\left[\rho D_{\n{n}} u  \right] & \,\,=\,\,\rho^- D_{\n{n}^{-}}  u^{-} +\rho^+D_{\n{n}^{+}} u^{+}  = \rho^-\nabla u^{-} \cdot \n{n}^{-}+ \rho^+\nabla u^{+} \cdot \n{n}^+, \quad \left[u\right]\,\,=\,\,u^{+}-u^{-},
\end{align*}
where $u^{\pm}\equiv u|_{\Omega^{\pm}}$ and $\n{n}^{\pm}$ is the unit outward normal to $\Omega^{\pm}$. We furthermore assume that $\rho^{+}\ge \rho^{-}>0$ are constants and that the interface $\Gamma$ is a {closed, simple and regular $\mathcal{C}^2$ curve with an arc-length parameterization $\n{X}$}.

There has been a recent surge in the development of finite element methods for interface problems. See for instance \cite{MR3338673, MR3047906, MR2571349, MR3051411, MR3264337, MR2820966, MR2738930, MR2192480, MR1941489, MR2864671, MR2377272, MR3268662,MR2981355, MR3248049, MR3218337, MR2684351}, to name a few. {Among the articles where the discretization is based on meshes not aligned with the interface}, most
of the methods focus on low contrast problems and only a few address the high contrast problems ($\rho^{+}/\rho^{-} \gg 1$).  For example, Burman et al. \cite{MR3051411} introduced an {unfitted Nitsche's method with averages and stabilization techniques for arbitrarily high-contrast problems, presenting bounds for the condition number of the stiffness matrix, although a rigorous error} analysis was not given in that paper. {Another example, is given by Chu et al. \cite{MR2684351} that uses multiscale techniques to build basis functions, an approach that seems well suited for high curvature problems (e.g. inclusions completely contained in a triangle).  However, in regions where the curvature of the interface is small it appears that their main a priori estimate degenerates (see Theorem 3.9 in \cite{MR2684351}), forcing them to refine the mesh on those regions in order to be aligned with the interface.  In our approach we do not need mesh refinements to make the triangulations aligned with the interface, however, we do not address high curvature problems.}

In order to put our contribution in context, let us explain two popular finite element approaches for problem \eqref{Problem}. The first approach is to double the degrees of freedom on triangles that intersect the interface and then add penalty terms to weakly enforce the continuity across the interface, see for example   \cite{MR3051411}. {Burman et al.} \cite{MR3051411}  demonstrated that in addition to penalizing the jumps across $\Gamma$ it is necessary to add a flux stabilization term. This method is the so-called stabilized unfitted Nitsche's method. The stabilization term  penalizes the jumps of the gradient on edges that belong to triangles that intersect the interface. As {Burman et al.} \cite{MR3051411} showed, in order to obtain a method that is robust with respect to diffusion contrast and robust with respect to the way $\Gamma$ cuts triangles, this type of penalization is necessary. The second common approach, and the one we focus in this paper, is to define local piecewise polynomial finite element spaces on triangles that intersect the interface $\Gamma$ (see for instance \cite{MR3218337,MR2681555, MR2377272, MR2684351, MR2864671, MR2740492}).
The basis functions are constructed by having them satisfy the continuity of the solution and the continuity of the flux strongly across $\Gamma$. Unlike the
unfitted Nitsche's method {that weakly imposes the interface conditions, in this approach the flux conservation
and the continuity of the solution are enforced strongly, for example at certain points on $\Gamma$, without requiring
stabilization terms on $\Gamma$. This is an important and distinguishable feature of the Immersed Interface Method using a Finite Element formulation (Immersed Finite Element Methods, see \cite{MR2018791})}. These basis functions are defined locally on each triangle, {and therefore} they are naturally discontinuous across edges of the triangulation. Namely, Adjerid et al. in \cite{MR3218337}  proposed to penalize jumps of the trial functions across the edges.  Similarly, {Lin et al.} in \cite{MR3338673} added similar penalty terms and proved optimal error estimates. However, in {their analysis they do not consider high-contrast problems}.

In this paper we follow this approach, defining local basis functions that are piecewise polynomials on each side of triangles that are cut by $\Gamma$. However, an additional stabilization term is added as compared to the methods of Adjerid et al. \cite{MR3218337} and {Lin et al.} \cite{MR3338673}, allowing us to prove {error} estimates that are independent of the contrast $\rho^{+}/\rho^{-}$. By stabilization term we refer to a penalization of the jumps of the normal derivatives of the approximation across the edges that belong to triangles intersecting $\Gamma$. This idea was used before by {Burman et al.} \cite{MR3051411}, however here we use different stabilization parameters (and of course different basis functions).  Roughly speaking, the reason this flux
stabilization is important for high contrast problems, is that one does not want to move estimates from $\Omega^+$ to $\Omega^-$ because $\rho^+$ could be much larger than $\rho^-$. However, triangles that are cut by $\Gamma$ might have a thin part in  $\Omega^+$ and therefore inverse estimates might be affected. By adding the jumps of derivatives we can transfer the estimates to a neighboring triangle that will have a larger portion in $\Omega^+$.

{In order to prove error estimates independent of the contrast $\rho^+/\rho^-$ we assume $H^2$ regularity of $u$ on both $\Omega^+$ and $\Omega^-$}. To be more precise, the error estimate presented in this paper for the energy norm (\revblue{see \eqref{defV}}) is
\begin{align}\label{boundV}
 \|u-u_h\|_V \,\, \le& \,\,C\, {h} \left\{ \sqrt{\rho^-} (\|D u\|_{L^2(\Omega^-)}+\|D^2 u\|_{L^2(\Omega^-)}) \right. \\
& \,\, \left. \quad\,+ \sqrt{\rho^+} (\|D u\|_{L^2(\Omega^+)} +\|D^2u\|_{L^2(\Omega^+)})  \right\}. \nonumber
\end{align}
{Consequently,} assuming that $\Omega$ is convex and using regularity estimates (see \cite{MR2684351}) we will have the result
\begin{equation} \label{boundV2}
\|u-u_h\|_V \,\, \le \,\,C\, \frac{h}{\sqrt{\rho^-}} \|f\|_{L^2(\Omega)}.
\end{equation}
{In addition, using a duality argument we prove the estimate} $\|u - u_h\|_{L^2(\Omega)} \leq C (h^2/\rho^-) \|f\|_{L^2(\Omega)}$. The constants in the estimates depend on the geometry, including the curvature of $\Gamma$.

The outline of the paper is as follows. We formulate the discrete
problem as finding $u_h \in V_h$ such that $a_h(u_h,v_h) = (f,v_h), \, \forall
v_h \in V_h$. The discrete space $V_h$ is introduced in {S}ection \ref{SpaceVh}
and the bilinear form $a_h(\cdot,\cdot)$ in Section \ref{FEMsub}. In {S}ection
\ref{localappro}, fundamental results on element-wise weighted $L^2$ and $H^1$ norm
approximation for the space $V_h$ are established. {Coercivity and continuity of the bilinear form  $a_h$ are studied in Section \ref{coercivity} and Section \ref{continuitybin}, respectively. In particular for the continuity, we note that the use of an augmented norm is necessary for the analysis due to the presence of the penalty terms involving flux jumps.}
In Section \ref{apriori} the bound (\ref{boundV}) is established by estimating the
approximation error and the consistency error across $\Gamma$ and across
elements near $\Gamma$. {The error estimate in the $L^2$ norm is also proved in this section. {S}ection \ref{extensionsrelated} is devoted to present extensions of the method to three dimensions, discuss related methods}, {and state some concluding remarks}. Finally, in Section  \ref{Numericssection} we provide numerical experiments {corroborating our theoretical findings}. {An Appendix containing technical proofs and a computational consideration is also included}.

\vskip2mm

\section{The finite element method}\label{sectionfem}
\subsection{Notation and local finite element space} \label{SpaceVh}
In this section we present a finite element method for problem (\ref{Problem}) using piecewise linear polynomials.

We next develop notation.  Let $\mathcal{T}_h$, $0 < h <1$ be {an admissible family of triangulations of $\Omega$ (conforming)}, with $ \overline{\Omega} = \cup_{T\in \mathcal{T}_h} \overline{T}$ and the elements $T$ are mutually disjoint. Let $ h_T$ denote the diameter of the element $T$ and $h = \max_{T} h_T$. We let $\Eh$ be the set of all edges of the triangulation.  We  adopt the convention that edges $e$, elements $T$, sub-edges $e^\pm\,:=\,e\cap \Omega^\pm$, sub-elements $T^\pm\,:=\, T\cap \Omega^\pm$ and sub-regions $\Omega^\pm$ are open sets, and we use the over-line symbol to refer to their closure.

Let $\mathcal{T}_h^{\Gamma}$ denotes the set of triangles $T \in \mathcal{T}_h$ such that ${T}$ intersects $\Gamma$. We let $\EhG$ be the set of all the three
edges of triangles in $\mathcal{T}_h^\Gamma$.  {Since $\Gamma$ is $\mathcal{C}^2$ we have that $\|\n{X}''\|_{L^\infty} < \infty $ and we define the maximum curvature $\kappa:=\|\n{X}''\|_{L^\infty} $. Our analysis, will be valid when $h$ is sufficiently small.  To make this precise,  we will use the concept of a tubular neighborhood whose existence is a standard result in   differential geometry; see \cite{MR0394451} Section 2-7, Proposition 3.}

{
\begin{lemma}\label{tubularlemma}{(Existence of $r$-tubular neighborhood)}
Let $\Gamma$ be a regular, simple, $\mathcal{C}^2$  curve. For every $x\in \Gamma$ consider the line segment $N_{x}(r)$ of length $2r$ centered at $x$ and perpendicular to $\Gamma$ at $x$. Define the tubular neighborhood of radius $r$ of $\Gamma$ by $Tub(r) = \cup _{x\in\Gamma} N_{x}(r)$.  Then, there exists $r>0$ such that for any two points  $x,y\in \Gamma,\,x\neq y$, the line segments $N_{x}(r)$ and $N_{y}(r)$ are disjoint.
\end{lemma}
}

From now on we will work under the following assumptions.
\begin{Assmptn}\label{Assumption}
{Given Lemma \ref{tubularlemma} we make the following assumptions {on the triangulation}:
\begin{enumerate}
\item We assume that the triangulation is shape-regular, see \cite{MR1278258}.
\item We assume $h< {r}/{2}$ where $r$ is the radius of the tubular neighborhood of $\Gamma$.
\item The interface intersects the boundary of an element at most twice and at different edges.
\end{enumerate}}
\end{Assmptn}
{It is well known that $r \le {1}/{\kappa}$, and hence by our assumption $h < {1}/{(2\kappa)}$.  The radius $r$ also bounds from below how close the curve $\Gamma$ comes from self-intersecting (e.g. consider a dumbbell with a thin middle section). We make use of these assumptions to prove some of the technical lemmas (see Lemmas \ref{twotri}, \ref{compare}, \ref{TGammalessLT}).  }

 In addition, we define an element patch $\omega_{T}$ of a triangle $T$,  its restriction to  $\Omega^\pm$ and its intersection with $\Gamma$ by
\begin{equation*}
\wT\,:=\,\mathrm{Int}\big\{\bigcup_{K\in \mathcal{T}_h} \overline{K}\,:\, \overline{K}\cap\overline{T}\,\neq\,\emptyset \big\},\quad  \wT^{\pm}\,:=\,\wT\cap \Omega^\pm, \quad \wTG\,:=\,\wT\cap \Gamma,
\end{equation*}
\revblue{where ``$\mathrm{Int}$'' denotes the interior of the set.}

The introduction of notation for the patch will be relevant in the proof of the interpolation error further forward. We first need to build our local finite element space (on each element).
To this end, {we let $x_0$ be the midpoint with respect to the arc-length on the curve segment}
$\TG:=T \cap \Gamma$. We note that the midpoint choice is a preference of the authors, the proofs below hold for any $x_0 \in \TG$.
Let $L_T$  be the line segment inside $T$ which is tangent to $\Gamma$
at $x_0$. {We define by $\n{t}^\pm$  the unit tangent vector to $\Gamma$ by a $90^{\circ}$
clockwise rotation of $\n{n}^{\pm}$}. Figure \ref{figure1} illustrates
the definitions and notations  introduced above.
\begin{SCfigure}[][!htbp]
\caption{Illustration of our notation on an element  $T\in \mathcal{T}_h^\Gamma$.\vskip4mm}
\includegraphics[scale=.3]{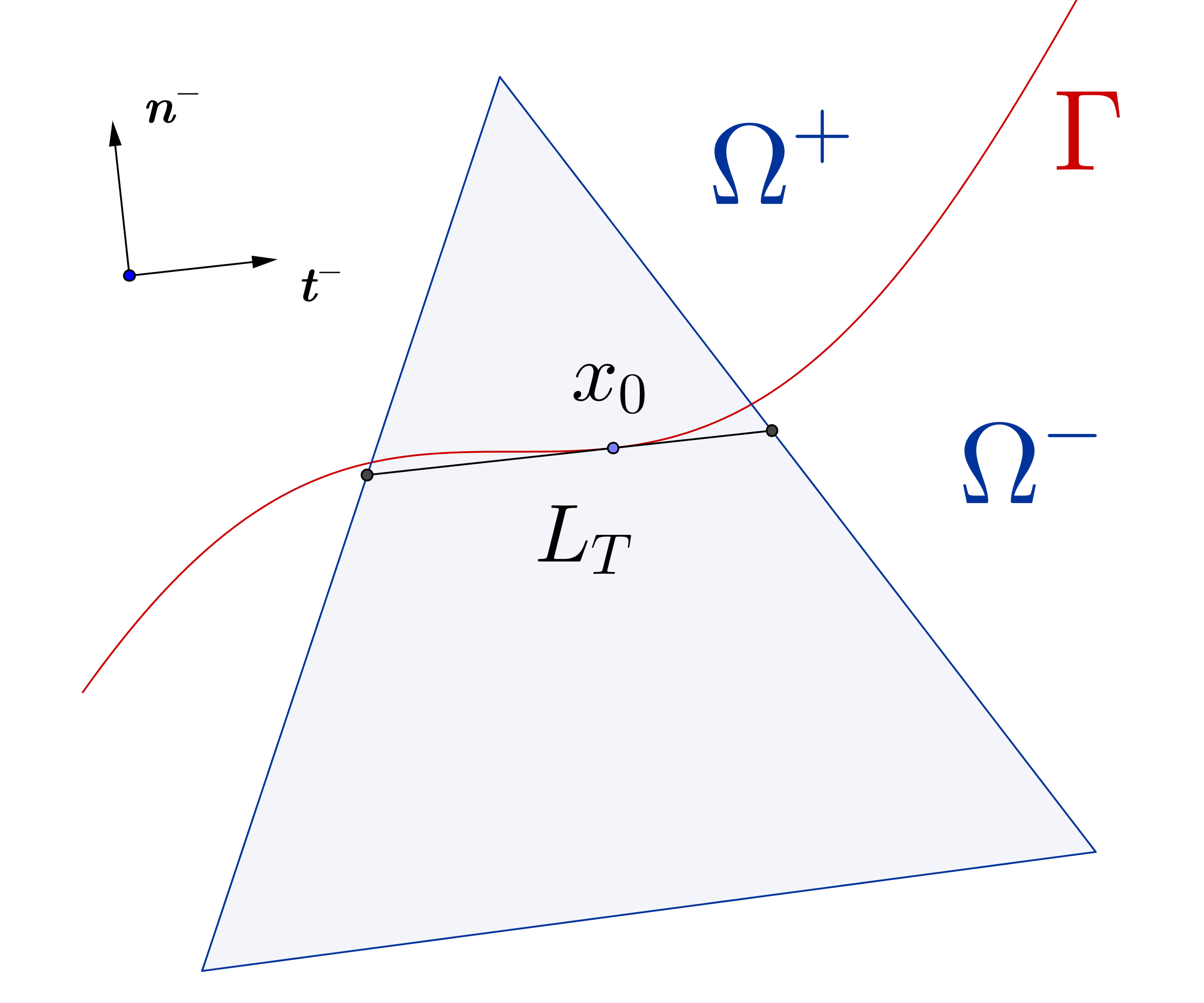}
\label{figure1}
\end{SCfigure}

In order to define our finite element space, we will need the following lemma.

\begin{lemma} \label{Lemma1}
{Consider the operator $\C:\mathbb{P}^1(T^+)\rightarrow \mathbb{P}^1(T^-)$ defined by}
\begin{alignat}{1}
{\C}(v)(x_0) &\,\,:=\,\, v(x_0), \label{eqUpsilon1} \\
  ( D_{\n{t}^+_0} {\C}(v) )(x_0) &\,\,:=\,\,  ( D_{\n{t}^+_0} v )(x_0),  \label{eqUpsilon2} \\
\rho^-( D_{\n{n}^+_0} {\C}(v) )(x_0) &\,\,:=\,\,
\rho^+( D_{\n{n}^+_0} v )(x_0),\label{eqUpsilon3}
\end{alignat}
where ${\n{n}^\pm_0} = {\n{n}^\pm}(x_0)$ and ${\n{t}^\pm_0} = {\n{t}^\pm}(x_0)$. {Then, $\Upsilon$ is well defined.}
\end{lemma}
Note that the exact solution $u^\pm$ satisfies the transmission
conditions \eqref{eqUpsilon1}, \eqref{eqUpsilon2} and \eqref{eqUpsilon3} on all points over $\Gamma$.

Next, given $T\in \mathcal{T}_h^\Gamma$ and for each $v
\in \mathbb{P}^1(T^+)$ we can consider the unique corresponding function
\[
G(v)=
\begin{cases}
v, & \text{ on } T^{+}, \\
{\C}(v), & \text{ on } T^{-}.
\end{cases} \]
Let  $\text{ span } \{v_1, v_2, v_3\}$ be a basis for $\mathbb{P}^1(T)$
restricted to $T^+$.  Then we define the local finite element space
\begin{equation}\label{localFEspace}
S^1(T)\,\,= \,\,
\begin{cases}
\text{span } \big\{G(v_1), G(v_2), G(v_3) \big\},& \mbox{ if } T\in \mathcal{T}_h^\Gamma, \\
\qquad\qquad\qquad\mathbb{P}^1(T), & \mbox{ if } T \in \mathcal{T}_h \backslash \mathcal{T}_h^\Gamma.
\end{cases}
\end{equation}
{An explicit construction of basis functions of the local space $S^1(T)$ is given in Appendix \ref{AppendixS1T}}.
The global finite element space is defined by
\begin{equation*}
V_h\,:=\, \left\{ v\,:\, v|_T \in S^1(T),\, \forall T \in \mathcal{T}_h,\, v \text{ is continuous across all edges in } \EhGn \right\}.
\end{equation*}

\subsection{Finite element method} \label{FEMsub}

We begin this section {by} introducing some standard discontinuous finite element notation for jumps and averages.

For a piecewise smooth function $v$ with support on $\mathcal{T}_h$,  we define its average and jump across an interior edge $e\in\mathcal{E}^h\backslash\partial \Omega$, shared by elements $T_1$ and $T_2$, as
\begin{equation*}
\big\{v\big\} \,\,= \,\,\frac{v|_{T_{1}}+v|_{T_{2}}}{2},\qquad\llbracket v \rrbracket\,\,=\,\, v|_{T_{1}}\n{n}_{1} \,+\,v|_{T_2}\n{n}_{2},
\end{equation*}
where $\n{n}_{1}$ and $\n{n}_{2}$ are the outward pointing  unit normal vectors to $T_{1}$ and $T_{2}$, respectively.
Similarly, if $\n{\tau}$ is a vector-valued function, piecewise smooth on $\mathcal{T}_h$, its average and normal jump across an interior edge $e$ are defined as
\begin{equation*}
\big\{\n{\tau}\big\} \,\,= \,\,\frac{\n{\tau}|_{T_{1}}+\n{\tau}|_{T_{2}}}{2},\qquad \llbracket \n{\tau} \rrbracket\,\,=\,\, \n{\tau}|_{T_{1}}\cdot\n{n}_{1} \,+\,\n{\tau}|_{T_2}\cdot\n{n}_{2}.
\end{equation*}

Next we introduce the finite element approximation to problem \eqref{Problem}.  First, we define the space $V$ as the union of the broken Sobolev spaces
\begin{equation}\label{VandH2h}
V := H^2_h(\Omega^+)\cup H^2_h(\Omega^-),\quad\mbox{where}\quad H_h^{2}(\Omega^{\pm})= \{v: v|_{T^\pm} \in H^2(T^\pm), \text{ for all } T \in \mathcal{T}_h\}.
\end{equation}
We can then define the bilinear form $a_h: V\times V \rightarrow \mathbb{R}$  and the linear functional $(f,\cdot):V \rightarrow \mathbb{R}$ by
\begin{alignat}{1}
a_h(w, v)\,\,:=\,\,& \int_{\Omega} \rho \nabla_h w \cdot \nabla_h v - \sum_{e \in \EhG}  \int_e  \left(\big\{\rho \nabla_h v \big\}
\cdot \llbracket w \rrbracket+   \big\{\rho \nabla_h  w \big\} \cdot \llbracket v \rrbracket\right) \label{a_h}\\
& +  \sum_{e \in \EhG} \left(\frac{\gamma}{|e^-|} \int_{e^-} \rho^- \llbracket w \rrbracket \cdot \llbracket v \rrbracket+ \frac{\gamma}{|e^+|} \int_{e^+} \rho^+ \llbracket w \rrbracket\cdot \llbracket v \rrbracket \right) \nonumber   \\
& + \sum_{e \in \EhG} \left({|e^-|} \int_{e^-} \rho^-\llbracket \nabla_h  v\rrbracket \, \llbracket \nabla_h  w \rrbracket +{|e^+|} \int_{e^+} \rho^+\llbracket \nabla_h  v\rrbracket\,\llbracket \nabla_h  w\rrbracket \right) ,\nonumber
\end{alignat}
\begin{equation*}
(f,v)\,\,:=\,\, \sum_{T\in \mathcal{T}_h} \Big( \int_{T^-} f^- v \, +\,\int_{T^+} f^+ v  \Big),
\end{equation*}
for a penalty parameter $\gamma\,>\,0$. {The discrete gradient operator $\nabla_h$ is piecewise defined on $T^{\pm}$ for an element $T\in\mathcal{T}_h$ by
\begin{equation*}
\nabla_h v|_{T^{\pm}} = D v = \nabla v.
\end{equation*}

Finally, the finite element approximation solves: Find $u_h \in V_h$ such that
\begin{equation}\label{fem}
a_h(u_h, v)\,\,=\,\,(f,v), \quad \text{ for all } v \in V_h.
\end{equation}

Note that {in \eqref{a_h}} we not only  penalize the jumps of the function but also the normal jumps of the first derivatives across edges. This will allow us to prove coerciveness and a priori error estimates independent of the contrast of the coefficients and also independent of how small $T^{+}$ or $T^-$ might be. Finally, we like to stress that in \eqref{a_h} only the normal derivative jumps are penalized, not the tangential derivative jumps.

\begin{figure}
\begin{center}
\includegraphics[scale=.38]{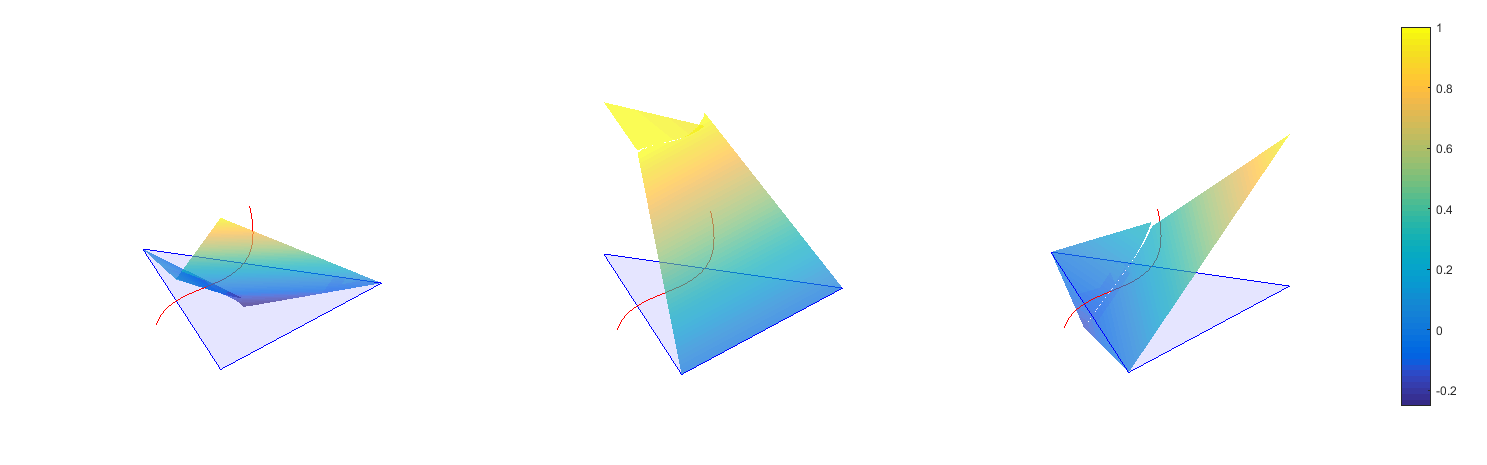}\\
\end{center}
\caption{Illustration of basis functions on triangle {in} Figure \ref{figure1} with $\rho^+ = 100$  and $\rho^- = 1$.}
\label{plotbasis}
\end{figure}

\section{Local Approximation} \label{localappro}
In this section we show that our finite element space has optimal local approximation properties.  {Henceforth we will keep track, as much as possible, on how constants depend on the maximum curvature $\kappa$, and in general in the geometry of sub-domains. When we use standard results (e.g. regularity and extensions lemmas) we just state that the constants depend on the geometry without explicitly saying how they scale with quantities such as  curvature and the radius of the tubular neighborhood $r$}. In order to define an interpolation operator onto $V_h$  we first state an extension result.
{Consider $u^\pm \in H^2(\Omega^\pm)$ (with $u^\pm  \equiv 0$ on $\partial \Omega^\pm \cap  \partial \Omega $). Then, there exists  a constant $C$ and extensions
$u_E^\pm \in H^2(\Omega)$ with the following properties
\begin{alignat}{2}
u_E^\pm&\,\,=\,\,u^\pm \qquad \text{ in } \Omega^\pm, \\
\|D u_E^\pm\|_{L^2(\Omega)}+\|D^2 u_E^\pm\|_{L^2(\Omega)} & \le {C}( \|D u^\pm\|_{L^2(\Omega^\pm)}+\|D^2 u^\pm\|_{L^2(\Omega^\pm)} ).
\label{extension1}
\end{alignat}
This result follows from Theorem 7.25 in \cite{MR737190} and applying Poincare's inequalities.  {Considering that we will only require the extensions $u_E^\pm$  to be defined on patches of elements on $\ThG$, in fact,} we only need the following  more {\it local} bound  which could lead to a better geometric constant:
\begin{equation}
\|D u_E^\pm\|_{L^2(\Omega^{\pm}\bigcup Tub(2h))}+\|D^2 u_E^\pm\|_{L^2(\Omega^{\pm}\bigcup Tub(2h))}  \le \Ce ( \|D u^\pm\|_{L^2(\Omega^\pm)}+\|D^2 u^\pm\|_{L^2(\Omega^\pm)} ).\label{extension}
\end{equation}
For the sake of simplicity, we do not prove  here how $\Ce$ depends on the geometric constants (e.g. $\kappa$ and $r$).}

\begin{Def} \label{Definition1}
Let $u^{\pm} \in H^2(\Omega^\pm)$.  For each $ T \in \mathcal{T}_h^{\Gamma}$ we
define $I_T u  \in S^1(T)$. In fact,
we define $I_T u $ on all $\wT$
\[
I_T u=
\begin{cases}
I^+_T u,  & \text{ in } \wT^{+}  \\
I^-_T u,  & \text{ in } \wT^{-},
\end{cases} \]
where $I_T^\pm$ are defined satisfying the following conditions
\begin{equation}\label{defIT}
\left\{
  \begin{array}{rcl}
(I^-_T u)(x_0) \,\,\,:=&(J_T u^+_E)(x_0)&=:\,\,\, (I^+_T u) (x_0) \\
  (D_{\n{t}^+_0}  I^-_T u)(x_0)\,\,\,:=&  (D_{\n{t}^+_0} (J_T u^+_E))(x_0)&=:\,\,\,(D_{\n{t}^+_0} I^+_T u)(x_0)    \\
\rho^- (D_{\n{n}^+_0} I^-_T u)(x_0) \,\,\,:=& \rho^- (D_{\n{n}^+_0} J_T u^-_E)(x_0) &=:\,\,\, \rho^+ (D_{\n{n}^+_0} I^+_T u)(x_0).
  \end{array}
\right.
\end{equation}
Here  $J_T$ is the $L^2$ projection operator onto $\mathbb{P}^1(\wT)$ ({note that $\wT \subset Tub(r)$ since we are assuming $2h \le r$}).

{If $T$ does not belong to  $\mathcal{T}_h^{\Gamma}$, we consider the sets $\Omega_{h}^{\pm} = \cup\{T\in\mathcal{T}_{h}: T \subset \Omega^{\pm}\}$, \revblue{and we define $I_T u = I_{\text{sz}} u^\pm$, for $T\subset \Omega^{\pm}$, where  $I_{\text{sz}} u^\pm$ is the Scott-Zhang interpolant of $u^\pm$ defined on $\Omega_h^\pm$.}  However, we need to modify the interpolant at every vertex $x$ (if one exists) that is an endpoint of an edge  $e$ satisfying  $e \subset \partial \Omega_h^+\cap \partial \Omega_h^-$ (e.g. $e \subset \Gamma$). In this case, we define $I_{sz} u^\pm (x)$ to be the average of $u^\pm$ on the edge $e$. Note that such an $x$ might correspond to two edges.  Either edge will work for the definition of  $I_{sz} u^\pm (x)$.   This will give $I_{sz} u^+= I_{sz} u^-$ for every edge  $e \subset \partial \Omega_h^+\cap \partial \Omega_h^-$ since $u^+=u^-$ on such an edge.  }

Consequently, we define the interpolation operator $I_h$ onto the finite element space $V_h$ as the restriction of the local interpolation operator $I_T$, i.e.
\begin{equation}\label{I_h}
I_h u|_T \,\,=\,\, I_T(u)|_T,\qquad \mbox{for all } T\in \mathcal{T}_h.
\end{equation}
\end{Def}

\revblue{The local interpolant $I_T u$ was constructed so that Lemma \ref{localcurve} holds which in turn is the main tool to prove the crucial Lemma \ref{normestimate}. The interpolator $I_T u$ was designed so that if we multiply the
  whole inequalities \eqref{localcurve1} and \eqref{localcurve2} by  $\rho^-$ and $\rho^+$, respectively, only terms of the form $\rho^\pm \|D ^j u_E^\pm\|_{L^2(\omega_T)}$  (after using $\rho^- \le \rho^+$) will appear.}

 \revblue{Since the proof of Lemma  \ref{localcurve} is quite involved, we first  prove a local energy stability of the interpolant whos proof is much easier and that motivates the definition of $I_h$. More specifically, we prove the following lemma. }
 
 \begin{lemma}
 It holds,
 \begin{equation*}
\revblue{ \rho^+ \|\nabla(I_T^+ u)\|_{L^2(\omega_T^+)} + \rho^-\|\nabla(I_T^- u)\|_{L^2(\omega_T^-)}  \le C (\rho^- \| \nabla u_E^-\|_{L^2(\omega_T)}+ \rho^+ \| \nabla u_E^+\|_{L^2(\omega_T)}).}
 \end{equation*}
 \revblue{where $C$ is independent of $\rho^\pm$}.
 \end{lemma}
 \begin{proof}
 \revblue{Clearly we have}
 \begin{equation*}
 \revblue{ \|\nabla(I_T^+ u)\|_{L^2(\omega_T^+)}^2 = \|D_{\n{t}_0^+}(I_T^+ u)\|_{L^2(\omega_T^+)}^2 + \|D_{\n{n}_0^+}I_T^{+} u\|_{L^2(\wT^{+})}^2}.
 \end{equation*}
 \revblue{Using that $D_{\n{t}_0^+}(I_T^+ u)$ and $D_{\n{n}_0^+}(I_T^{+} u)$ are constant and using the definition of the interpolant \eqref{defIT} we have }
\begin{equation*}
 \revblue{ \rho^+ \|\nabla(I_T u)\|_{L^2(\omega_T^+)} \le C (\rho^+ \|D_{\n{t}_0^+}(J_T u_E^+)\|_{L^2(\omega_T^+)} + \rho^{-}\|D_{\n{n}_0^+}(J_T u_E^-)\|_{L^2(\wT^{+})}).}
 \end{equation*}
 \revblue{We trivially have }
 \begin{equation*}
 \revblue{ \rho^+ \|\nabla(I_T u)\|_{L^2(\omega_T^+)} \le C (\rho^+ \|\nabla (J_T u_E^+)\|_{L^2(\omega_T)} + \rho^{-}\|\nabla (J_T u_E^-)\|_{L^2(\omega_T)}).}
 \end{equation*}
 \revblue{Define $c^{\pm}=\frac{1}{|\omega_T|} \int_{\omega_T}  u_E^{\pm}$ and note that $J_T c^\pm=c^\pm $. We then see  }
\begin{equation*}
 \revblue{ \rho^+ \|\nabla(I_T u)\|_{L^2(\omega_T^+)} \le C(\rho^+ \|\nabla (J_T (u_E^+-c^+))\|_{L^2(\omega_T)} + \rho^{-}\|\nabla (J_T (u_E^--c^-))\|_{L^2(\omega_T)}).}
 \end{equation*}
 \revblue{Using an inverse estimate, the stability of the $L^2$ projection $J_T$, and Poincare's inequality we get }
 \begin{equation*}
 \revblue{ \rho^+ \|\nabla(I_T u)\|_{L^2(\omega_T^+)} \le C (\rho^+ \|\nabla u_E^+\|_{L^2(\omega_T)} + \rho^{-}\|\nabla u_E^-\|_{L^2(\omega_T)}).}
 \end{equation*}
 \revblue{In a similar way we can prove the estimate for $\rho^-\|\nabla(I_T^- u)\|_{L^2(\omega_T^-)} $ where we only need to use $\rho^- \le \rho^+$. }
 
 \end{proof}

\begin{lemma}\label{lemma-s}
Consider $T\in \ThG$. Then, for every $v \in \mathbb{P}^1(\wT)$ the following bound holds
\begin{equation*}
  h_T^j\|D^j v\|_{L^{2}(\wT)} \,\,\le\,\, C \left(h_T |v(x_0)| + h^2_T |D_{\n{n}^+_0} v(x_0)|
+ h^2_T |D_{\n{t}^+_0} v(x_0)|\right)\revblue{,\quad j=0,1}.
\end{equation*}
\end{lemma}
\begin{proof}
Using the Cauchy-Schwarz inequality one has
\begin{equation*}
h_T^j\|D^j v\|_{L^{2}(\wT)}\,\, \le\,\, C \, h_T^{j+1}\|D^j v\|_{L^{\infty}(\wT)}\revblue{,\quad j=0,1.}
\end{equation*}
Using the fact that $v$ is a linear function, the lemma follows.
\end{proof}

We next prove a fundamental result of this paper, a local approximation
property on the space $S^1(T)$.

\begin{lemma} \label{localcurve}
Let $u^\pm \in H^2(\Omega^\pm)$ satisfying the interface conditions
$\left[u\right]=0$ and $\left[\rho D_{\n{n}}  u  \right]=0$. Then, for any $ T \in \mathcal{T}_h^{\Gamma}$ and
$j\,=\,0,\, 1$, the following bounds hold:
\begin{alignat}{1}
h_T^{j}  \|D^j(u^{-}-I_T^{-}u)\|_{L^2(\wT^{-})}\,\, \le\,\,  C {(1+\kappa)}\,
h_T^{2} \left(\| D u_E^{-}\|_{L^2 (\wT)} +
\| D^2 u_E^{-}\|_{L^2 (\wT)}+  \| D^2 u_E^{+}\|_{L^2 (\wT)} \right) \label{localcurve1}
\end{alignat}
and
\begin{alignat}{1}
 \label{localcurve2}h_T^{j} \|D^j(u^{+}-I_T^{+}u)\|_{L^2(\wT^{+})}\,\, \le& \,\,
C \, {(1+ \kappa)} h_T^{2}   \Big(  \| D u_E^{+}\|_{L^2 (\wT)} + \| D^2 u_E^{+}\|_{L^2 (\wT)}  \\
&\,\, \quad +\frac{\rho^-}{\rho^+}\| D^2 u_E^{-}\|_{L^2 (\wT)}\Big). \nonumber
\end{alignat}
\end{lemma}

\begin{proof} First note that by adding and subtracting  $J_T u_E^\pm$ to $u_E^\pm -I^\pm_Tu$,  using the triangle inequality, and by means of the following well known approximation property for $\varphi\in H^{2}(\wT)$
\begin{equation}\label{proj}
h_T^j\|D^j (\varphi-J_T\varphi)\|_{L^{2}(\wT)} \,\,\le\,\, C \, h^2_T
\|D^2 \varphi\|_{L^2(\wT)} \quad\mbox{ for } j \,=\, 0,\,1,
\end{equation}
the proof of the lemma reduces to estimate
\[
h_T^j \| D^j w_T^\pm \|_{L^2(\wT^{\pm})},\quad \mbox{where}\quad
w_T^\pm := J_T u_E^\pm- I_T^\pm u,~~ \mbox{for}~~ w_T^\pm\in {\mathbb{P}^1(\wT)}.
\]
According to the definition of $I_T$ \eqref{defIT}, and denoting
$\n{n}^{\pm}_0 = \n{n}^{\pm}(x_0)$ and $\n{t}^{\pm}_0 = \n{t}^{\pm}(x_0)$,
we have
\begin{equation*}\label{wt-def}
\left\{
  \begin{array}{ll}
w_T^-(x_0) & \,\,= \,\, J_T u_E^-(x_0) - J_T u_E^+(x_0), \\
(D_{\n{t}^+_0} w_T^-)(x_0) &\,\, =\,\, (D_{\n{t}^+_0}  J_T u_E^-)(x_0) - (D_{\n{t}^+_0}  J_T u_E^+)(x_0), \\
(D_{\n{n}^+_0} w_T^-)(x_0) &\,\,= \,\,0,
  \end{array}
\right.
\end{equation*}
and
\begin{equation*}\label{wt+def}
\left\{
  \begin{array}{ll}
w_T^+(x_0) &\,\,= \,\,0, \\
(D_{\n{t}^+_0} w_T^+)(x_0) &\,\,=\,\, 0, \\
\rho^+(D_{\n{n}^+_0}      w_T^+)(x_0) &\,\,= \,\,\rho^+(D_{\n{n}^+_{0}}  J_T u_E^+)(x_0) -
\rho^-(D_{\n{n}^+_0}  J_T u_E^-)(x_0).
  \end{array}
\right.
\end{equation*}
Then, using Lemma \ref{lemma-s} and the values derived for $w_{T}^{\pm}$ we have
\begin{equation} \label{boundeq1}
h_T^j \| D^j w_T^- \|_{L^2(\wT^{-})} \,\,\le\,\,
 C  \left(h_T |w_T^-(x_0)| + h_T^2 |D_{\n{t}^+_0}w_T^-(x_0)|\right),
\end{equation}
and
\begin{equation} \label{boundeq2}
h_T^j \| D^j w_T^+ \|_{L^2(\wT^{+})}\,\, \le\,\,
 C  h_T^2
| D_{\n{n}^+_0}w_T^+(x_0)|.
\end{equation}
We proceed by bounding the two terms in \eqref{boundeq1} and the term in \eqref{boundeq2}, separately. For the first term in \eqref{boundeq1},
we use that $u$ is continuous on the interface \eqref{Problem:c}, in particular on $ \wTG$, and we apply the triangle inequality
\begin{alignat*}{1}
|w^-_T(x_0)|\,\,= \,\,|(J_T u_E^- -J_T u_E^+ )(x_0)|\,\, \le &\,\, \|J_Tu_E^--J_Tu_E^+\|_{L^\infty(\wTG)} \\
 \le & \,\, \|J_Tu_E^--u_E^-\|_{L^\infty(\wTG)} + \|J_Tu_E^+-u_E^+\|_{L^\infty(\wTG)}.
\end{alignat*}
{Note that $ch_T\leq |\wTG| \leq C h_T$, where the second inequality follows from Lemma \ref{TGammalessLT}}. By means of a Sobolev inequality in one dimension, we have
\begin{equation*}
\|J_Tu_E^--u_E^-\|_{L^\infty(\wTG)}  \,\,\le\,\, C \left(\frac{1}{\sqrt{h_T}} \|J_Tu_E^--u_E^-\|_{L^2(\wTG)}+ \sqrt{h_T} \|D(J_Tu_E^--u_E^-)\|_{L^2(\wTG)}\right),
\end{equation*}
consequently, using a trace inequality we obtain
\begin{alignat*}{1}
\|J_Tu_E^--u_E^-\|_{L^\infty(\wTG)}\,\, \le\,\, & C \left(h_T^{-1} \|J_Tu_E^{-}-u_E^-\|_{L^2(\wT)}+ \|D(J_Tu_E^--u_E^-)\|_{L^2(\wT)}\right) \\
 & \quad +C \, h_T \|D^2u_E^-\|_{L^2(\wT)}.
\end{alignat*}
Hence, using the approximation property \eqref{proj}, we get
\begin{equation*}
\|J_T u_E^--u_E^-\|_{L^\infty(\wTG)} \le C h_T \|D^2u_E^-\|_{L^2(\wT)}.
\end{equation*}
Analogously, we can show that
\begin{equation*}
\|J_T u_E^+-u_E^+\|_{L^\infty(\wTG)} \le C h_T \|D^2u_E^+\|_{L^2(\wT)}.
\end{equation*}
Therefore, we have the bound for the first term in \eqref{boundeq1}
\begin{equation}\label{inq:101}
 h_T|w_T^-(x_0)| \,\,\le \,\, C\, h_T^2 \left(\|D^2u_E^-\|_{L^2(\wT)} + \|D^2u_E^+\|_{L^2(\wT)}\right).
\end{equation}
We now turn to the second term in \eqref{boundeq1}. We use the fact that $D_{\n{t}^{\revblue{+}}_{0}}w_T^-$ is constant on $ \wTG$ to obtain
\begin{equation*}
|D_{\n{t}^{+}_{0}}w_T^-(x_0)| \,\,=\,\,
  |\wTG|^{-1/2} \|D_{\n{t}^{+}_0}w_T^-\|_{L^2( \wTG)}\,\,
\leq\,\, C h_T^{-1/2}  \|D_{\n{t}^{+}_0}(J_T u_E^- -J_T u_E^+)\|_{L^2( \wTG)}.
\end{equation*}
Using the identities {, denoting $\n{t}^+ = \n{t}^+(x)$ and $\n{n}^+ = \n{n}^+(x)$, for $x\in \wTG$}
\begin{eqnarray}
D_{\n{t}^+_0} u^\pm_E &=& ( \n{t}^+_{0}\cdot \n{t}^+) D_{{\n{t}}^+} u^\pm_E
+ ( \n{t}_0^+\cdot \n{n}^+) D_{{\n{n}}^+}u^\pm_E ,\\
D_{{\n{t}}^+}  u^+_E& =& D_{{\n{t}}^+}  u^-_E\label{interface1},\\
D_{\n{n}^+} u^+_E &= &\frac{\rho^-}{\rho^+} D_{\n{n}^+} u^-_E\label{interface2},
\end{eqnarray}
we have
\begin{alignat*}{1}
 \|D_{\n{t}^{+}_0} w_T^-\|_{L^2( \wTG)} \,\,& =\,\,     \big\|D_{\n{t}^{+}_0} (J_T u_E^- -u_E^- +  u_E^+ -J_T u_E^+) +
(1 - \frac{\rho^-}{\rho^+}) ( \n{t}_{0}^+\cdot \n{n}^+) D_{{\n{n}}^+} u_E^-
\big\|_{L^2( \wTG)} \\
& \le\,\,
\|D (J_T u_E^- -u_E^-)\|_{L^2(\wTG)} +
\|D (J_T u_E^+ -u_E^+)\|_{L^2(\wTG)} \\
& \quad +
(1 - \frac{\rho^-}{\rho^+}) \|( \n{t}_0^+\cdot \n{n}^+) D u_E^-
\|_{L^2( \wTG)}.
\end{alignat*}
For the first two terms in the previous bound we use a trace inequality to obtain
\begin{equation*}
\|D (J_T u_E^\pm - u_E^\pm)\|_{L^2(\wTG)} \leq C \left(\frac{1}{\sqrt{h_T}}
\|D (J_T u_E^\pm  - u_E^\pm)\|_{L^2(\wT)} + \sqrt{h_T}
\|D^2 (J_T u_E^\pm -u_E^\pm)\|_{L^2(\wT)} \right),
\end{equation*}
and for the third term we use that $|\n{t}_0^-\cdot \n{n}^-| \leq \|\n{X}''\|_{L^{\infty}}|\wTG| = {\mathcal{O}(\kappa h_T)}$
on $ \wTG$, $\rho^- \leq
\rho^+$, and  a
trace inequality to obtain
\begin{equation*}
(1 - \frac{\rho^-}{\rho^+}) \|( \n{t}_0^+\cdot \n{n}^+) D u_E^-
\|_{L^2( \wTG)} \leq C\, {\kappa}  \left(\sqrt{h_T} \|D u_E^-\|_{L^2(\wT)} +
h_T^{3/2} \|D^2 u_E^-\|_{L^2(\wT)}\right).
\end{equation*}
Hence, applying approximation property \eqref{proj} we obtain
\begin{equation} \label{inq:102}
h_T^2  |D_{\n{t}^+_0}w_T^-(x_0)|
\,\,\leq\,\, C {(1+\kappa)} h_T^2\left(\|D u_E^-\|_{L^2(\wT)} +
 \|D^2 u_E^-\|_{L^2(\wT)}) + \|D^2 u_E^+\|_{L^2(\wT)}\right).
\end{equation}
If we combine the inequalities (\ref{inq:101}) and (\ref{inq:102})
with (\ref{boundeq1}), we arrive at the first result of the lemma, inequality \eqref{localcurve1}.

Now we estimate the term in \eqref{boundeq2}. We need
to bound
\begin{equation*}
|D_{\n{n}^{+}_0}w_T^+(x_0)| = | \wTG|^{-1/2} \|D_{\n{n}^{+}_0}w_T^+\|_{L^2( \wTG)}
\leq C h_T^{-1/2}  \|D_{\n{n}^{+}_0}(J_T u_E^+ - \frac{\rho^-}{\rho^+}
J_T u_E^-)\|_{L^2(\wTG)}.
\end{equation*}
Similarly to the previous bound, w{e} note that
\[
D_{\n{n}^+_0} u^\pm_E = ( \n{n}_{0}^+\cdot \n{t}^+) D_{{\n{t}}^+} u^\pm_E
+ ( \n{n}_{0}^+\cdot \n{n}^+) D_{{\n{n}}^+}u^\pm_E,
\]
and then using  \eqref{interface1} and \eqref{interface2} we obtain
\begin{alignat*}{1}
 & \|D_{\n{n}^{+}_{0}} w_T^+\|_{L^2( \wTG)}  =    \big\|D_{\n{n}^{+}_{0}} \big(J_T u_E^+ - u_E^+ - \frac{\rho^-}{\rho^+} (J_T u_E^- -
 u_E^-)\big)+  (1 - \frac{\rho^-}{\rho^+}) ( \n{n}_{0}^+\cdot \n{t}^+) D_{{\n{t}}^+} u_E^+
\big\|_{L^2( \wTG)}.
\end{alignat*}
The remaining of the proof is similar as above and we obtain
\[
h_T^2  |D_{\n{n}^+_0}w_{T}^+(x_0)|
\leq C {(1+\kappa)} h_T^2\left(\frac{\rho^-}{\rho^+} \|D^2 u_E^-\|_{L^2(\wT)}  +
\|D u_E^+\|_{L^2(\wT)} + \|D^2 u_E^+\|_{L^2(\wT)}\right).
\]
If we combine this inequality with (\ref{boundeq2}), we arrive at our second
result (\ref{localcurve2}).
\end{proof}

\section{Coercivity and Continuity of Bilinear Form}
The aim of this section is to prove coercivity and continuity of the bilinear
form $a_h(\cdot,\cdot)$ defined in \eqref{a_h}.

\subsection{Coercivity of Bilinear Form} \label{coercivity}
We define the following energy norm  $\|\cdot\|_V\,:\, V \rightarrow \mathbb{R}^+_{0}$
\begin{alignat}{1} 
\|v\|_V^2 \,\,=\,\,&\|\sqrt{\rho} \nabla_h v\|_{L^2(\Omega)}^2+ \sum_{e \in \EhG} \left(\frac{\rho^-}{|e^-|} \|\llbracket v\rrbracket\|_{L^2(e^-)}^2 + \frac{\rho^+}{|e^+|} \|\llbracket v\rrbracket\|_{L^2(e^+)}^2\right)  \label{defV}
\\
& + \sum_{e \in \EhG} \left(\rho^-|e^-| \|\llbracket \nabla_h  v\rrbracket \|_{L^2(e^-)}^2 + \rho^+|e^+| \|\llbracket \nabla_h v \rrbracket\|_{L^2(e^+)}^2\right) \nonumber.
\end{alignat}

In order to prove coercivity we will need the following lemma.
\begin{lemma} \label{twotri}
Let $e = \mathrm{Int}(\partial T_1\cap\partial T_2)\in\EhG$, for $T_1, T_2\in \Th$. Then, there exists a constant $\theta>0$ such that
\begin{equation*}
| e^\pm |^2\,\, \le\,\, \theta  \max_{i=1,2} |T_i^\pm|.
\end{equation*}
The constant $\theta$ depends on the shape regularity of the triangulation.
\end{lemma}
\begin{proof}
See Appendix \ref{twotriproof}.
\end{proof}

\begin{lemma}\label{lemma:coercive}(Coercivity)
If $\gamma$ is large enough (depending on shape regularity of triangulation) there exists a constant $c>0$, {independent of $h$, $\rho^-$ and $\rho^+$,} such that
\begin{equation}\label{coercive}
c\|v\|_V^2 \,\,\le\,\,  a_h(v,v), \quad \text{ for all } v\in V_h.
\end{equation}
\end{lemma}
\begin{proof}
Let $v \in V_h$, then
\begin{alignat*}{1}
a_h(v, v)\,\,=&\,\, \|\sqrt{\rho} \nabla_h v\|_{L^2(\Omega)}^2 -2 \sum_{e \in \EhG}  \int_e \{\rho \nabla_h v \}\cdot \llbracket v \rrbracket \\
& +  \sum_{e \in \EhG} \left(\frac{\gamma}{|e^-|}  \rho^- \| \llbracket v\rrbracket \|_{L^2(e^-)}^2+\frac{\gamma}{|e^+|}  \rho^+ \| \llbracket v\rrbracket \|_{L^2(e^+)} ^2\right)   \\
& +\sum_{e \in \EhG} \left(\rho^-|e^-| \| \llbracket \nabla_h  v\rrbracket \|_{L^2(e^-)}^2 + \rho^+|e^+| \| \llbracket \nabla_h  v\rrbracket \|_{L^2(e^+)}^2\right).
\end{alignat*}
To prove the lemma, it is enough to bound the non-symmetric  term
\begin{equation*}
\int_e \{\rho \nabla_h v \} \cdot \llbracket v \rrbracket = \int_{e^-} \{\rho \nabla_h v \}\cdot \llbracket v \rrbracket+ \int_{e^+} \{\rho \nabla_h v \} \cdot\llbracket v \rrbracket.
\end{equation*}

Let $T_1, T_2 \in \mathcal{T}_h$ be such that $e= \mathrm{Int}(\partial {T}_1 \cap \partial {T}_2)$ and set $T_e=T_1 \cup T_2$. Without loss of generality assume $\displaystyle \revblue{|T_2^-|}=\max_{i=1,2} |T_i^-|$. Then, according to Lemma \ref{twotri}
\begin{equation}\label{edge}
|e^-|^2 \le \theta  |T_2^-|.
\end{equation}
Then, we see that
\begin{alignat*}{1}
 \int_{e^-} \{\rho \nabla v \} \cdot \llbracket v \rrbracket\,\,=&\,\,\int_{e^-} \frac{\rho^-}{2} \Big( (D_{\n{n}_1} v)\,  (v|_{T_1}-v|_{T_2})
 + (D_{\n{n}_2} v) \,( v|_{T_2}-v|_{T_1})\Big) \\
=& \int_{e^-} \rho^- (D_{\n{n}_2} v)\,  (v|_{T_2}-v|_{T_1})
 + \frac{\rho^-}{2} (D_{\n{n}_1}v + D_{\n{n}_2}v)  ( v|_{T_1}-v|_{T_2}) \\
=& \int_{e^-} \rho^- (D_{\n{n}_2} v)\,  (v|_{T_2}-v|_{T_1})
 + \frac{\rho^-}{2} (\llbracket \nabla_{h} v \rrbracket)\,  ( v|_{T_1}-v|_{T_2}).
\end{alignat*}
Using the fact that $v$ is a linear function on  $T_2^-$  and using \eqref{edge} we have

\begin{equation}\label{Dn1bound}
\big|2\int_{e^-} \rho^- (D_{\n{n}_2} v)\,  (v|_{T_2}-v|_{T_1})\big|  \le {C}\, \|\rho^{-} {D} v\|_{L^2(T_2^-)} \frac{1}{|e^-|^{1/2}} \|\llbracket v \rrbracket\|_{L^2(e^-)},
\end{equation}
where ${C}$ depends on $\theta$. Therefore, we have
\begin{equation*}
\big|2\int_{e^-} \rho^- (D_{\n{n}_2} v)\,  (v|_{T_2}-v|_{T_1})\big|  \le \epsilon \rho^{-}\| {\nabla_{h}} v\|_{L^2(T_e^-)}^2  + \frac{\rho^- {C}^2}{ \epsilon |e^-|} \|\llbracket v \rrbracket \|_{L^2(e^-)}^2,
\end{equation*}
for any $\epsilon >0$. Furthermore, we have
\begin{equation*}
 \big|\int_{e^-} {\rho^-} (\llbracket \nabla_h v \rrbracket)\, ( v|_{T_1}-v|_{T_2})\big|  \le  \epsilon \rho^- |e^-| \| \llbracket \nabla_h v\rrbracket \|_{L^2(e^-)}^2  + \frac{\rho^-}{\epsilon |e^-|} \|\llbracket v \rrbracket\|_{L^2(e^-)}^2.
\end{equation*}
Collecting the last two estimates gives
\begin{equation*}
\big|2\int_{e^-} \{\rho \nabla_h v \} \cdot \llbracket v \rrbracket \big| \le \epsilon \rho^- (|e^-| \| \llbracket \nabla_h v \rrbracket \|_{L^2(e^-)}^2+ \| \nabla_h v\|_{L^2(T_e^-)}^2)  + \frac{({C}^2+1)\rho^-}{\epsilon |e^-|} \|\llbracket v \rrbracket \|_{L^2(e^-)}^2.
\end{equation*}
Likewise, we can bound the integral over $e^+$ to get a combined result
\begin{alignat*}{1}
\big|2\int_{e} \{\rho \nabla_h v \}\cdot \llbracket v \rrbracket\big| \le &\epsilon  (\|\sqrt{\rho} \nabla_h v\|_{L^2(T_e)}^2+  \rho^{-} |e^-| \| \llbracket \nabla_h v \rrbracket \|_{L^2(e^-)}^2+\rho^+ |e^+| \| \llbracket \nabla_h v \rrbracket \|_{L^2(e^+)}^2) \\
 & + \frac{({C}^2+1)\rho^-}{\epsilon |e^-|} \|\llbracket v \rrbracket\|_{L^2(e^-)}^2+  \frac{({C}^2+1) \rho^+ }{\epsilon |e^+|} \|\llbracket v \rrbracket\|_{L^2(e^+)}^2.
\end{alignat*}
Summing over all edges $e\in \EhG$ we get
\begin{alignat*}{1}
\big|2 \sum_{e \in \EhG}  \int_e \{\rho \nabla_h v \} \cdot\llbracket v \rrbracket\big| \le & \sum_{e \in \EhG} \epsilon  (\| \sqrt{\rho} \nabla_h v\|_{L^2(T_e)}^2+  \rho^- |e^-| \| \llbracket \nabla_h  v\rrbracket \|_{L^2(e^-)}^2+\rho^+ |e^+| \| \llbracket \nabla_h v \rrbracket \|_{L^2(e^+)}^2) \\
&  +\sum_{e \in \EhG} \frac{({C}^2+1) \rho^-}{\epsilon |e^-|} \|\llbracket v \rrbracket \|_{L^2(e^-)}^2+  \frac{({C}^2+1)\rho^+}{\epsilon |e^+|} \|\llbracket v \rrbracket \|_{L^2(e^+)}^2.
\end{alignat*}
{Finally,} the result follows by choosing $\epsilon =1/2$ and then choosing $\gamma=\frac{({C}^2+1)}{\epsilon}+1${.}
\end{proof}

\subsection{Continuity of bilinear form} \label{continuitybin}

Next we prove continuity of the bilinear form. To do this we define the  augmented norm
\begin{equation} \label{augmentednorm}
\|v\|_{W}^2\,\,=\,\,\|v\|_V^2+ \sum_{{e \in \EhG}}   (\rho^- \|\{ \nabla_h v\} \cdot \n{n}\|_{L^2(e^-)}^2| e^- | +  \rho^+ \|\{ \nabla_h v\} \cdot \n{n}\|_{L^2(e^+)}^2 |e^+ |).
\end{equation}

\begin{lemma} (Continuity) \label{lemma:continuity}
Suppose that ${v,\, w\in V}$. Then, there exists a constant {$C>0$, independent of $h$, $\rho^-$ and $\rho^+$,} such that
\begin{equation*}
a_h(w,v)\,\, \le\,\,  C \, \|w\|_W  \,  \|v\|_W.
\end{equation*}
Additionally, if $w \in V$ and $v \in V_h$ we have
\begin{equation}\label{continuity}
a_h(w,v) \,\,\le\,\,  C \, \|w\|_W  \,  \|v\|_V.
\end{equation}
\end{lemma}
\begin{proof}
We give a sketch of the proof by bounding  each term of $a_h(\cdot,\cdot)$ in
\eqref{a_h} separately. The first term can easily be bounded by Cauchy-Schwarz inequality
\begin{equation*}
  \int_{\Omega} \rho \,\nabla_h v \cdot \nabla_h w \,\,\leq \,\, \|\sqrt{\rho} \nabla_h v \|_{L^2(\Omega)} \, \|\sqrt{\rho} \nabla_h w\|_{L^2(\Omega)}.
\end{equation*}
{The first part of the second term} can be written as
\begin{equation*}
\int_e \left\{ \rho\nabla_{{h}} v \right\} \cdot\llbracket w \rrbracket \,\, = \,\, \int_{e^-} \left\{ \rho^{-}\nabla_{{h}} v \right\} \cdot \llbracket w \rrbracket + \int_{e^+} \left\{ \rho^{+}\nabla_{{h}} v \right\} \cdot \llbracket w \rrbracket.
\end{equation*}
Using the Cauchy-Schwarz inequality one has
\begin{alignat*}{1}
&\int_{e^\pm} \left\{ \rho^{\pm}\nabla_{{h}} v \right\} \llbracket w \rrbracket  \,\, \leq \,\,\left( \sqrt{| e^\pm|}  \sqrt{\rho^\pm} \|\{ \nabla_{{h}} v\} \cdot \n{n}\|_{L^2(e^\pm)}  \right)  \frac{\sqrt{\rho^\pm}}{\sqrt{| e^\pm |}} \| \llbracket w\rrbracket \|_{L^2(e^\pm)}.
\end{alignat*}
Let $e=\mathrm{Int}(\partial T_1\cap\partial T_2)$. If $v \in V_h$ then $\nabla_{{h}} v|_{T_i^\pm}$  is constant for each $i=1,2$. Then,  we can use Lemma \ref{twotri} to get
\begin{equation*}
\sqrt{| e^\pm|}  \sqrt{\rho^\pm} \|\{ \nabla_h v\} \cdot \n{n}\|_{L^2(e^\pm)} \,\,\le\,\,  \sqrt{| e^\pm|} \sqrt{\rho^\pm} \|  \llbracket  \nabla_h v \rrbracket   \|_{L^2(e^\pm)} + C \sqrt{\rho^\pm} (\| \nabla_h v\|_{L^2(T_1^\pm)}+ \| \nabla_h v\|_{L^2(T_2^\pm)})
\end{equation*}
where $C$ here depends on $\theta$. The third term can be bounded by
\begin{equation*}
\int_{e^\pm} \rho ^{\pm} \llbracket  w \rrbracket \cdot \llbracket  v \rrbracket \,\, \leq \,\, \rho^{\pm} \| \llbracket w\rrbracket \|_{L^2(e^\pm)}  \| \llbracket v\rrbracket \|_{L^2(e^\pm)}.
\end{equation*}
Finally, the fourth term can be bounded by Cauchy-Schwarz inequality
\begin{equation*}
\int_{e^\pm} \rho ^{\pm} \llbracket \nabla_h w \rrbracket \, \llbracket \nabla_h v \rrbracket \,\, \leq \,\, \rho^{\pm} \| \llbracket \nabla_h v\rrbracket \|_{L^2(e^\pm)}  \| \llbracket \nabla_h w\rrbracket \|_{L^2(e^\pm)}.
\end{equation*}
The proof is complete summing over the edges, using finite overlapping of the elements associated to the edges and using arithmetic-geometric mean inequality.
\end{proof}

\section{A priori error estimates} \label{apriori}
The purpose of this section is to prove a priori error estimates for the method defined in \eqref{fem}-\eqref{a_h} in the energy and $L^2$ norm.

\subsection{Energy error estimates}

\begin{theorem} \label{maintheorem} (A priori energy error estimate)
Let $u$ be the solution to problem \eqref{Problem} and $u_h \in V_h$ be its finite element approximation solution of \eqref{fem}. Then, there exists $C>0$, independent of $h$,  $\rho^-$ and $\rho^+$, such that
\begin{equation*}\label{mainerror}
\|u-u_h\|_V  \le C\, h (1+\kappa) \left(\sqrt{\rho^-}(||Du||_{L^2 (\Omega^-)}+ \|D^2u\|_{L^2(\Omega^-)})+
  \sqrt{\rho^+} ( || D u||_{L^2(\Omega^+)} + || D^2 u||_{L^2(\Omega^+)})  \right).
\end{equation*}
{The constant $C$ depends on $\Ce$ from \eqref{extension}}.
\end{theorem}

\begin{proof}
Note that $\|u - u_h\|_V \leq \|u - I_h u\|_V + \|I_h u-u_h\|_V$.
Since $d_h:=I_h u-u_h \in V_h$, coercivity of $a_h(\cdot,\cdot)$ (see \eqref{coercive}) gives
\begin{equation}\label{aux4}
c\|d_h\|_{V}^2 \,\,\le\,\, a_h(d_h, d_h)\,\,=\,\,a_h(I_h u-u, d_h)+ a_h(u-u_h, d_h).
\end{equation}
Using \eqref{continuity} we obtain, for $\epsilon>0$
\begin{equation}\label{aux3}
a_h(I_h u-u, d_h)\,\, \le\,\, C \, \|I_h u-u\|_{W} \|d_h\|_{V} \le C \left(\frac{\epsilon}{2} \|d_h\|_{V}^2+ \frac{1}{2\epsilon}\|I_h u-u\|_{W}^2\right).
\end{equation}
Choosing $\epsilon$ sufficiently small we have
\begin{equation}\label{aux100}
\|d_h\|_{V}^2 \,\,\le\,\, C \, \|I_h u-u\|_{W}^2+  a_h(u-u_h, d_h).
\end{equation}
The proof of the theorem follows by checking {below} Lemma \ref{normestimate}
(approximation error in $W$-norm) and Lemma \ref{inconsistency} (inconsistency error).
\end{proof}

We next prove the missing parts of Theorem \ref{maintheorem}. We need an estimate for interpolant $I_h$ (see \eqref{I_h}) in the augmented norm $W$, and to do so, we first need to prove a trace inequality that goes from a part of boundary to the interior of the domain.

{
\begin{lemma}\label{smalltolarge}
Let $T\in \mathcal{T}_{h}^{\Gamma}$ and let $e$ be an edge of $T$. Suppose that $w \in  {H^2(\wT^{-})\cup H^2(\wT^{+})}$ and $|e^{\pm}|>0$. Then, there exists a constant such that
\begin{equation*}
\frac{1}{|e^\pm|} \|w\|_{L^2(e^\pm)}^2\,\, \le\,\,  C  \left(\frac{1}{h_T^2} \|w\|_{L^2(\wT^{\pm})}^2+ \|D w\|_{L^2(\wT^{\pm})}^2+ h_T^2 \|D^2 w\|_{L^2(T^{\pm})}^2\right).
\end{equation*}
\end{lemma}
\begin{proof}
Consider the case where $e\cap\Gamma\neq \emptyset$. Let $e_1$ be an interior edge of $\wT$, contained in $\wT^{-}$ and connected by one node to $e^{-}$. Edge $e_1$ exists thanks to the $r$-tubular neighborhood assumption with $2 h< r$. Note that  $c h_T \le |e_1|$ for a constant $c$ that only depends on the shape regularity of the mesh. Using the fundamental theorem of calculus  we can show (see Lemma 3 of \cite{MR1273155} for a similar two dimensional result)
\begin{equation*}
\frac{1}{|e^-|} \|w\|_{L^2(e^-)}^2 \,\,\le\,\, C( |e^-|\, \|{D} w\|_{L^2(e^-)}^2 +\frac{1}{h_T}\|w\|^2_{L^2(e_1)}+ h_T\|{D} w\|_{L^2(e_1)}^2).
\end{equation*}
Noting that $|\wT^-| \ge C h_T^2$, and using standard trace inequalities gives the result. The exact same argument would apply for $e^+$. If $e\cap\Gamma= \emptyset$, $e^{-} = e$ we apply trace inequality from $e$ to $\wT^-$.
\end{proof}
}

\begin{lemma} (Best Approximation Error Estimate) \label{normestimate}
It holds,
\begin{equation*}
\|I_h u-u\|_{W} \,\,\le\,\, C\, h\, {(1+\kappa)} \left(\sqrt{\rho^-} (\|Du\|_{L^2 (\Omega^-)}+ \|D^2u \|_{L^2(\Omega^-)})+
  \sqrt{\rho^+}( \| D u\|_{L^2(\Omega^+)} + \| D^2 u\|_{L^2(\Omega^+)}) \right).
\end{equation*}
{Here the constant $C$ depends on the constant $\Ce$ from \eqref{extension}. }
\end{lemma}

\begin{proof}
We bound each term of $\|u-I_h u\|_W$, see (\ref{augmentednorm}) and \eqref{defV}, separately. Using Lemma \ref{localcurve} we easily have
\begin{alignat*}{1}
& \|\sqrt{\rho}\nabla_h(u-I_h u)\|_{L^2(\Omega)}^2 \,\, \leq\,\,
C \, \sum_{T \in \mathcal{T}_h \backslash \mathcal{T}_h^\Gamma} h_T^2
\left(\rho^- \|D^2 u_E^-\|_{L^2(T)}^2+ \rho^+ \|D^2 u_E^+\|_{L^2(T)}^2 \right) \\
& + C\, \sum_{T \in \mathcal{T}_h^\Gamma} h_T^2 {(1+\kappa)^2} \,
 \left({\rho^-} (\|D u_E^-\|^2_{L^2 (\wT)} + \|D^2 u_E^-\|^2_{L^2 (\wT)})
+ {\rho^+}( \|D u_E^+\|^2_{L^2 (\wT)}+  \|D^2 u_E^+\|^2_{L^2 (\wT)}) \right)  \\
&\leq\,\,  C\, h^2 \, {(1+\kappa)^2}  \left(\rho^- (\|Du\|_{L^2(\Omega^-)}^2+
\|D^2u\|_{L^2(\Omega^-)}^2) + \rho^+ (  \|Du\|_{L^2(\Omega^+)}^2 + \|D^2u\|_{L^2(\Omega^+)}^2 )\right) .
\end{alignat*}
In the last inequality we used \eqref{extension}. We also have used finite overlapping of the sets $\wT$. We next estimate the term
\begin{equation*}
\sum_{e \in \EhG} \left(\frac{\rho^-}{|e^-|} \|\llbracket u-I_h u\rrbracket\|_{L^2(e^-)}^2 + \frac{\rho^+}{|e^+|} \|\llbracket u-I_h u\rrbracket\|_{L^2(e^+)}^2\right).
\end{equation*}
Suppose that $e= \mathrm{Int}(\partial T_1 \cap \partial T_2)$.  Then, using Lemma \ref{smalltolarge} we have
\begin{alignat*}{1}
 & \frac{\rho^-}{|e^-|} \|\llbracket u-I_h u\rrbracket\|_{L^2(e^-)}^2 + \frac{\rho^+}{|e^+|} \|\llbracket u-I_h u\rrbracket\|_{L^2(e^+)}^2  \\
& \le C  \frac{\rho^-}{|e^-|} (\|u-I_{T_1} u\|_{L^2(e^-)}^2 +  \|u-I_{T_2} u\|_{L^2(e^-)}^2) +C \frac{\rho^+}{|e^+|} (\| u-I_{T_1}\|_{L^2(e^+)}^2+  \| u-I_{T_2}\|_{L^2(e^+)}^2)  \\
& \le C \rho^- \sum_{i=1}^2 \left(\frac{1}{h_{T_i}^2} \|u-I_{T_i} u\|_{L^2(\omega_{T_{i}}^{-})}^2+  \|D (u-I_{T_i} u)\|_{L^2(\omega_{T_i}^{-})}^2+ h_{T_i}^2 \|D^2(u-I_{T_i} u)\|_{L^2(\omega_{T_{i}}^{-})}^2\right)  \\
& \quad +  C  \rho^+ \sum_{i=1}^2 \left(\frac{1}{h_{T_i}^2} \|u-I_{T_i} u\|_{L^2(\omega_{T_i}^{+})}^2+ \|D (u-I_{T_i} u)\|_{L^2(\omega_{T_i}^{+})}^2+ h_{T_i}^2 \|D^2(u-I_{T_i} u)\|_{L^2(\omega_{T_{i}}^{+})}^2\right).
\end{alignat*}
If we now apply {L}emma \ref{localcurve} and use \eqref{extension}, we get
\begin{alignat*}{1}
&\sum_{e \in \EhG} (\frac{\rho^-}{|e^-|} \|\llbracket u-I_h u\rrbracket\|_{L^2(e^-)}^2 + \frac{\rho^+}{|e^+|} \|\llbracket u-I_h u\rrbracket\|_{L^2(e^+)}^2) \\
&\qquad \le
 C h^2 {(1+\kappa)^2}  \left(\rho^- (\|Du\|_{L^2(\Omega^-)}^2+
\|D^2u\|_{L^2(\Omega^-)}^2) + \rho^+ (\|Du\|_{L^2(\Omega^+)}^2+  \|D^2u\|_{L^2(\Omega^+)}^2)\right).
\end{alignat*}
{We next bound}
\begin{equation*}
\sum_{e \in \EhG} \left(\rho^-|e^-|\, \|\llbracket D (u-I_h u)\rrbracket\|_{L^2(e^-)}^2 + \rho^+|e^+| \,\|\llbracket D(u-I_h u) \rrbracket\|_{L^2(e^+)}^2\right).
\end{equation*}
Using a trace inequality to obtain
\begin{alignat*}{2}
 \rho^-|e^-| \,\|\llbracket \nabla_h(u-I_h u)\rrbracket&\|_{L^2(e^-)}^2 + \rho^+|e^+|\, \|\llbracket \nabla_h (u-I_h u) \rrbracket\|_{L^2(e^+)}^2 \\
&\le\,\,  \rho^-|e| \,\|\llbracket \nabla_h(u-I_h u)\rrbracket\|_{L^2(e^-)}^2 + \rho^+|e| \,\|\llbracket \nabla_h (u-I_h u) \rrbracket\|_{L^2(e^+)}^2 \\
&\le\,\, \rho^-  |{e}|\left(\|D( u-I_{T_1} u)\|_{L^2(e^-)}^2+\|D(u-I_{T_2} u)\|_{L^2(e^-)}^2\right)  \nonumber  \\
&\,\, \quad  + \rho^+|{e}| \left(\|D (u-I_{T_1} u) \|_{L^2(e^+)}^2+  \|D(u-I_{T_2} u) \|_{L^2(e^+)}^2\right) \nonumber  \\
&\le \,\, C  \rho^- \sum_{i=1}^2 \left( \|D (u-I_{T_i} u)\|_{L^2(\omega_{T_{i}}^{-})}^2+ h_{T_i}^2 \|D^2(u-I_{T_i} u)\|_{L^2(\omega_{T_{i}}^{-})}^2\right) \nonumber \\
& \,\, \quad+  C  \rho^+ \sum_{i=1}^2 \left( \|D(u-I_{T_i} u)\|_{L^2(\omega_{T_i}^{+})}^2+ h_{T_i}^2 \|D^2(u-I_{T_i} u)\|_{L^2(\omega_{T_i}^{+})}^2\right). \nonumber
\end{alignat*}
Again,  if we apply {L}emma \ref{localcurve} and \eqref{extension} we obtain
\begin{alignat*}{2}
& \sum_{e \in \EhG} \left(\rho^-|e^-| \|\llbracket \nabla_h  (u-I_h u)\rrbracket\|_{L^2(e^-)}^2 \right.+\left. \rho^+|e^+| \|\llbracket \nabla_h (u-I_h u) \rrbracket\|_{L^2(e^+)}^2\right)  \\
& \le \,\,   C h^2 {(1+\kappa)^2}  \left(\rho^- (\|Du\|_{L^2(\Omega^-)}^2+
\|D^2u\|_{L^2(\Omega^-)}^2) + \rho^+ (\|Du\|_{L^2(\Omega^+)}^2+   \|D^2u\|_{L^2(\Omega^+)}^2)\right).
\end{alignat*}
Finally, the last term is estimated by
\begin{alignat*}{1}
& \sum_{T \in \mathcal{E}_h}   (| e^- | \rho^- \|\{ \nabla_h (u-I_h u)\} \cdot \n{n}\|_{L^2(e^-)}^2+  | e^+ | \rho^+ \|\{ \nabla_h (u-I_h u)\} \cdot \n{n}\|_{L^2(e^+)}^2) \\
& \le  \,\,
 C h^2 {(1+\kappa)^2}  \left(\rho^- (\|Du\|_{L^2(\Omega^-)}^2+
\|D^2u\|_{L^2(\Omega^-)}^2) + \rho^+ ( \|Du\|_{L^2(\Omega^+)}^2+ \|D^2u\|_{L^2(\Omega^+)}^2)\right),
\end{alignat*}
exactly in the same way as it was done above.
\end{proof}

In order to bound the inconsistency term we need a technical lemma.
\begin{lemma}\label{compare}
Let $T \in \mathcal{T}_h^\Gamma$ and enumerate the three edges of $T$ by $e_1, e_2, e_3$. Then, there exists a constant $m$ such that
\begin{equation*}
|\TG| \,\,\le\,\, m \max_{i=1,2,3} |e_i^-| \quad\text{ and }\quad   |\TG| \,\,\le \,\,m \max_{i=1,2,3} |e_i^+|.
\end{equation*}
The constant $m$ {is independent of $r$ (\revblue{see Lemma \ref{tubularlemma}}), depends only on the shape regularity of the triangulation.}
\end{lemma}
\begin{proof} See Appendix \ref{compareproof}.
\end{proof}

Now we are able to establish the inconsistency error estimate. The  constant will be
independent of the contrast of the coefficients $\rho^-$ and $\rho^+$.
\begin{lemma} \label{inconsistency}(Inconsistency Error Estimate)
Let $u$ be the solution to \eqref{Problem} and $u_h \in V_h$ be its finite element approximation solution to \eqref{fem}. Then
for any $d_h \in V_h$, it holds
\begin{equation} \label{lemma11inq}
a_h(u-u_h,d_h) \,\,\le\,\, C\, h \, \kappa \left( \sqrt{\rho^-} \left(||D u||_{L^2 (\Omega^-)}+h \|D^2 u\|_{L^2(\Omega^-)} \right) + \sqrt{\rho^-} h \|D^2 u\|_{L^2(\Omega^+)} \right)
\|d_h\|_V.
\end{equation}
{The constant $C$ depends on the constant $\Ce$ from \eqref{extension}}.
\end{lemma}
\begin{proof}
First note that $a_h(u-u_h,d_h) = a_h(u,d_h) - (f,d_h)$. The inconsistency term is then given by
\begin{equation*}
a_h(u-u_h, d_h)\,\,=\,\, \sum_{T \in \mathcal{T}_h^{\Gamma}}  \int_{\TG} (\rho^- D_{\n{n}^-} u^-) [d_h] ,
\end{equation*}
where we have used that $\rho^- D_{\n{n}^-} u^-=\rho^+  D_{\n{n}^+}u^+$ on $\Gamma$. We also have used the notation $[d_h]=d_h^+-d_h^-$.   \revblue{Let $L_T$ be the tangent
line to $\Gamma$ at $x_0$ and note that $[d_h]$ is linear function on $L_T$ that vanishes $x_0$ as well as its first derivative}. Hence,  $[d_h] \equiv 0$ on the  line $L_T$. Since the distance between this line  and the curve $\TG$ is $O(s^2)$, where we set $s=|\TG|$, we have that $|[d_h](x)| \le C
{\|\n{X}''\|_{L^{\infty}}}\,s^2 |D [d_h]|$ for every  $x \in \TG$. Hence, we have
\begin{equation*}
 \int_{\TG} (\rho^- D_{\n{n}^-} u^-) [d_h] \,\,\le\,\, C\, {\kappa} \, s^2 \|\sqrt{\rho^-} D_{\n{n}^-} u_E^-\|_{L^2(\TG)} \sqrt{\rho^-} \| D[d_h]\|_{L^2(\TG)}.
\end{equation*}
For the moment, let us assume that following inequality holds
\begin{equation}\label{MT}
\sqrt{\rho^-} \sqrt{s} \|D[d_h]\|_{L^2(\TG)} \,\,\le\,\, C  \,M_T,
\end{equation}
where
\begin{alignat*}{1}
M_T\,\,= \,\,& \sum_{e \subset \partial T} \left( \sqrt{|e^+|} \|\sqrt{\rho^+} \llbracket {\nabla_h} d_h \rrbracket\|_{L^2(e^+)}+   \frac{1}{\sqrt{|e^+|}} \|\sqrt{\rho^+} \,d_h \|_{L^2(e^+)} + \|\sqrt{\rho^+} {\nabla_h}  d_h\|_{L^2(\wT^{+})}\right) \\
    &+ \sum_{e \subset \partial T} \left(\sqrt{|e^-|} \|\sqrt{\rho^-} \llbracket {\nabla_h}  d_h \rrbracket \|_{L^2(e^-)}+ \frac{1}{\sqrt{|e^-|}} \|\sqrt{\rho^-} \,d_h \|_{L^2(e^-)} + \|\sqrt{\rho^-} {\nabla_h}  d_h\|_{L^2(\wT^{-})}\right).
\end{alignat*}
Then,
\begin{equation*}
\int_{\TG} (\rho^- D_{\n{n}^-} u^-) [d_h]  \,\,\le\,\, C \,{\kappa} \, {s}^{3/2} \|\sqrt{\rho^-} D u_E^-\|_{L^2(\TG)} M_T.
\end{equation*}
Letting $B_s$ be the ball of radius $s$ centered at $x_0$ we can use the trace inequality to get
\begin{equation*}
\|\sqrt{\rho^-} D u_E^-\|_{L^2(\TG)}\,\, \le\,\, C (\frac{1}{\sqrt{s}}
\|\sqrt{\rho^-} D u_E^-\|_{L^2(B_s)}+ \sqrt{s} \|\sqrt{\rho^-}D^2 u_E^-\|_{L^2(B_s)}).
\end{equation*}
Hence,
\begin{equation*}
\int_{\TG} (\rho^- D_{\n{n}^-} u^-) [d_h] \,\,\le\,\, C s   (\|\sqrt{\rho^-} D u_E^-\|_{L^2(B_s)}+ s \|\sqrt{\rho^-}D^2 u_E^-\|_{L^2(B_s)}) M_T.
\end{equation*}
We see that \eqref{lemma11inq} follows after using that
\begin{equation*}
 \sum_{T \in \mathcal{T}_h^\Gamma} M_T^2\,\, \le\,\, C \, \|d_h\|_V^2,
\end{equation*}
the inequality \eqref{extension} and the fact that $s \le C \, h_T$ (see Lemma \ref{TGammalessLT}).

In order to complete the proof we need to prove \eqref{MT}. Firstly, by using the triangle inequality we have
\begin{equation*}
 \sqrt{\rho^-} \sqrt{s} \|D[d_h]\|_{L^2(\TG)} \,\,\le\,\,  \sqrt{\rho^-}\sqrt{s} \left(\|D d_h^-\|_{L^2(\TG)}+\| D d_h^+\|_{L^2(\TG)}\right).
\end{equation*}
Then by Lemma \ref{compare}, there exists an edge $e$ of $T$ such that
\begin{equation}\label{aux2}
s \,\,\le\,\, m |e^+|.
\end{equation}
Now let $e=\mathrm{Int}(\partial T\cap \partial K) $, for a $K \in \mathcal{T}_h$. Using that $D d_h^+$ is constant we get
\begin{alignat*}{1}
\| D d_h^+\|_{L^2(\TG)}\,\,&=\,\, \sqrt{s} |D d_h^+|\,\,=\,\,\frac{\sqrt{s}}{\sqrt{|e^+|}} \| D d_h^+\|_{L^2(e^+)} \,\, \le\,\, \sqrt{m} \| D d_h^+\|_{L^2(e^+)}.
\end{alignat*}
According to Lemma \ref{twotri}
\begin{equation}\label{aux1}
|e^+|^2 \,\,\le\,\, \theta \max\{|T^+|, |K^+|\}.
\end{equation}
First, suppose that $|T^+|=\max\{|T^+|, |K^+|\}$. Then,
\begin{equation*}
\| {\nabla_h}  d_h^+\|_{L^2(e^+)} \,\,= \,\,\frac{\sqrt{|e^+|}}{\sqrt{|T^+|}} \|D d_h^+\|_{L^2(T^+)}.
\end{equation*}
 Hence,
\begin{equation*}
\| {\nabla_h}  d_h^+\|_{L^2(e^+)}  \,\,\le \,\,\frac{\sqrt{\theta m}}{\sqrt{s}}  \|D d_h^+\|_{L^2(T^+)}.
\end{equation*}
On the other hand, if $|K^+|=\max\{|T^+|, |K^+|\}$ then we have
\begin{equation*}
\| {\nabla_h}  d_h^+\|_{L^2(e^+)} \,\,\le\,\, \| D (d_h^+ |_T -d_h^+ |_K) \|_{L^2(e^+)}+
\frac{\sqrt{\theta m}}{\sqrt{s}}
\|D d_h^+\|_{L^2(K^+)}.
\end{equation*}
We write $D (d_h |_T -d_h |_K)$ as its normal part and tangential part, then, use an inverse inequality for the tangential part on $e^+$, and have
\begin{equation*}
\| {\nabla_h}  (d_h^+ |_T -d_h^+ |_K) \|_{L^2(e^+)} \,\,\le\,\, \| \llbracket {\nabla_h} d_h \rrbracket \|_{L^2(e^+)} + \frac{C}{|e^+|} \| \llbracket d_h \rrbracket \|_{L^2(e^+)},
\end{equation*}
therefore,
\begin{equation*}
\| {\nabla_h}  d_h^+\|_{L^2(e^+)}\,\, \le\,\, \|\llbracket {\nabla_h}  d_h \rrbracket\|_{L^2(e^+)}+ \frac{\sqrt{\theta m}}{\sqrt{s}} \|D d_h^+\|_{L^2(K^+)}+\frac{C}{|e^+|} \| \llbracket d_h\rrbracket  \|_{L^2(e^+)}.
\end{equation*}
Combining the above inequalities we obtain
\begin{alignat*}{1}
 \sqrt{\rho^- s} \| D d_h^+\|_{L^2(\TG)} \,\,\le & \,\, \sqrt{m |e^+|}\sqrt{\rho^+} \|\llbracket {\nabla_h}  d_h \rrbracket \|_{L^2(e^+)}+\frac{C\sqrt{\rho^+}\sqrt{m}}{\sqrt{|e^+|}} \| \llbracket d_h \rrbracket  \|_{L^2(e^+)} \\
 &+ Cm \sqrt{\theta} ( \|\sqrt{\rho^+} D d_h\|_{L^2(T^+)}+ \|\sqrt{\rho^+} D d_h\|_{L^2(K^+)}),
 \end{alignat*}
where we have used $\rho^- \le \rho^+$, so
\begin{alignat*}{1}
 \sqrt{\rho^- s} \| D d_h^+\|_{L^2(\TG)} \,\,\le & \,\, C \sum_{e \subset \partial T}  ( \sqrt{m |e^+|}\sqrt{\rho^+} \|\llbracket {\nabla_h}  d_h \rrbracket \|_{L^2(e^+)}+\frac{\sqrt{\rho^+}\sqrt{m}}{\sqrt{|e^+|}} \| \llbracket d_h \rrbracket  \|_{L^2(e^+)} )\\
 &+  m \sqrt{\theta} \|\sqrt{\rho^+} \nabla_h d_h\|_{L^2(\wT^{+})}.
 \end{alignat*}

Using the same argument we can prove
\begin{alignat*}{1}
 (\rho^- s)^{1/2} \| D d_h^-\|_{L^2(\TG)} \,\, \le & \,\, C \sum_{e \subset \partial T}  ( \sqrt{m |e^-|}\sqrt{\rho^-} \|\llbracket D d_h \rrbracket \|_{L^2(e^-)}+\frac{\sqrt{\rho^-}\sqrt{m}}{\sqrt{|e^-|}} \| \llbracket d_h \rrbracket  \|_{L^2(e^-)} ) \\
 &+   Cm \sqrt{\theta} \|\sqrt{\rho^-} \nabla_h d_h\|_{L^2(\wT^{-})}.
 \end{alignat*}
Combining the two last inequalities we obtain {\eqref{MT}}.
\end{proof}

Now we would like to state a corollary of Theorem \ref{maintheorem}. We first need some regularity results. We start with a standard energy estimate.
\begin{proposition}
Let $u$ solve \eqref{Problem}, then the following energy estimate holds
\begin{equation}\label{energyinequality}
\sqrt{\rho^+}\|D u\|_{L^2(\Omega^+)} +\sqrt{\rho^-} \|D u\|_{L^2(\Omega^-)}\,\, \le\,\, \frac{C}{\sqrt{\rho^-}} \|f\|_{L^2(\Omega)}
\end{equation}
where $C$ is independent  of $\rho^\pm$.
\end{proposition}
At this point we recall that we are assuming that $\rho^- \le \rho^+$. In fact, a better energy estimate holds as we prove in the next proposition.

\begin{proposition}\label{prop3}
Let $u$ solve \eqref{Problem}, then the following energy estimate holds
\begin{equation}\label{energyinequality2}
\rho^+\|D u\|_{L^2(\Omega^+)} +\rho^- \|D u\|_{L^2(\Omega^-)}\,\, \le\,\, C \,  \|f\|_{L^2(\Omega)}.
\end{equation}
\end{proposition}
{The constant $C$ could depend on the geometry of the subdomains $\Omega^{\pm}$}.
\begin{proof}
We prove this estimate in the case that $\Omega^+$ is the inclusion (i.e. $\partial \Omega^+$ does not intersect $\partial \Omega$). The proof of the other case is similar. First, from the previous proposition we have
\begin{equation}\label{inq99}
\rho^- \|D u\|_{L^2(\Omega^-)}\,\, \le\,\, C\,  \|f\|_{L^2(\Omega)}.
\end{equation}

 Let $w=u^+-\frac{1}{|\Omega^+|}\int_{\Omega^+} u^+ dx$. Therefore, using Poincare's inequality we have
\begin{equation}\label{poincare}
\|w\|_{L^2(\Omega^+)} \le C \|D u\|_{L^2(\Omega^+)}.
\end{equation}
{By means of an extension result (see Lemma 6.37 in \cite{MR737190})}, there exists $v \in H_0^1(\Omega)$ such that
 $v|_{\Omega^+}=w|_{\Omega^+}$ such that
\begin{equation}\label{inq100}
\|v\|_{H^1(\Omega)} \le C \|w\|_{H^1(\Omega^+)}.
\end{equation}
{Applying the variational formulation of problem \eqref{Problem} we have that}
\begin{equation*}
\int_{\Omega^-} \rho^- D u \cdot D v+\int_{\Omega^+} \rho^+ D u \cdot D v dx=\int_{\Omega} f v.
\end{equation*}
{Multiplying by $\rho^+$ it follows that}
\begin{equation*}
{\|\rho^{+}D u\|_{L^2(\Omega^+)}^2}=\rho^+ \int_{\Omega^+} \rho^+ D u \cdot D v dx=\rho^+ \int_{\Omega} f v-\rho^+ \int_{\Omega^-} \rho^- D u \cdot D v
\end{equation*}
Therefore,  we have
\begin{equation*}
\|\rho^+ D u\|_{L^2(\Omega^+)}^2 \le \|f\|_{L^2(\Omega)} \| \rho^+ v\|_{L^2(\Omega)}+ \|\rho^- D u\|_{L^2(\Omega^-)} \|\rho^+ D v\|_{L^2(\Omega)} .
\end{equation*}
By \eqref{inq100} we have
  \begin{equation}
  \|\rho^+ D u\|_{L^2(\Omega^+)}^2  \le C (\|f\|_{L^2(\Omega)}+ \|\rho^- D u\|_{L^2(\Omega^-)}) \rho^+ \|w\|_{H^1(\Omega^+)}.
 \end{equation}
Hence, using \eqref{poincare} we get
  \begin{equation}
  \|\rho^+ D u\|_{L^2(\Omega^+)}^2  \le C (\|f\|_{L^2(\Omega)}+ \|\rho^- D u\|_{L^2(\Omega^-)})  \|\rho^+ D u\|_{L^2(\Omega^+)}.
 \end{equation}
The proof is complete after applying \eqref{inq99}.
\end{proof}

 We  will also need to state and $H^2$  regularity estimates which can { essentially be proved if one combines results in  \cite{MR2684351} and \cite{MR1929889}. Here we give an alternative argument using  the results in \cite{MR1929889} and classical regularity theory.}

\begin{proposition}\label{prop1}
In addition to the assumptions already made in this article, assume further that $\Omega$ is convex. Let $u$ solve \eqref{Problem}. Then, the following regularity estimate holds
\begin{equation}\label{H2regularity}
\rho^+\|D^2 u\|_{L^2(\Omega^+)} +\rho^-\|D^2 u\|_{L^2(\Omega^-)} \,\,\le\,\, \Cr \, \|f\|_{L^2(\Omega)},
\end{equation}
where $C$ is independent of $\rho^\pm$.
\end{proposition}
{The constant $\Cr$ could depend on the geometry of the subdomains $\Omega^{\pm}$}.
\begin{proof}
{
We consider two cases: \\
Case 1: $\Omega^-$ is the inclusion (i.e. $\partial \Omega^- \cap \partial \Omega =\emptyset$) \\
In this case, \eqref{H2regularity} follows from (2.35) in \cite{MR1929889}.}
\vskip2mm

\noindent Case 2: $ \Omega^+$ is the inclusion\\
From (2.36) in \cite{MR1929889} we have
\begin{equation}\label{prop1inq1}
{\rho^- \|u\|_{H^2(\Omega^-)} \le C \, \|f\|_{L^2(\Omega)}.}
 \end{equation}

 {Let $\bar{u}^{+}= \frac{1}{|\Omega^+|} \int_{\Omega^+} u^+ dx$. Let  $w=u^+-\bar{u}^{+}$. Note that $w$ satisfies
 \begin{alignat*}{3}
   -\rho^{+} \Delta w &= f \qquad &\mbox{in } & \Omega^{+}, \\
                                          \rho^+ \nabla  w  \cdot \n{n}&= \rho^+ \nabla u^+ \cdot \n{n}   \qquad          & \mbox{on } &\partial \Omega^+.
\end{alignat*}
 }
{
From a standard regularity result (see  (2.3.3.1) in \cite{MR775683}) we have
\begin{equation*}
{\rho^+ \|w\|_{H^2(\Omega^+)} \le C( \|f\|_{L^2(\Omega^+)}+ \|\rho^+ \nabla u^+ \cdot \n{n}\|_{H^{1/2}(\partial \Omega^+)} + \rho^+ \| w\|_{H^1(\Omega^+)}).}
\end{equation*}
Using the jump condition we have
\begin{equation*}
 \|\rho^+ \nabla u^+ \cdot \n{n}\|_{H^{1/2}(\partial \Omega^+)} = \| \rho^- \nabla u^- \cdot \n{n}\|_{H^{1/2}(\partial \Omega^+).}
\end{equation*}
By a standard trace inequality we have
\begin{equation*}
 \| \rho^- \nabla u^- \cdot \n{n}\|_{H^{1/2}(\partial \Omega^+)}  \le \rho^- \| u^-\|_{H^2(\Omega^-)}.
\end{equation*}
Which by \eqref{prop1inq1} gives
\begin{equation*}
\|\rho^+ \nabla u^+ \cdot \n{n}\|_{H^{1/2}(\partial \Omega^+)}  \le  C \, \|f\|_{L^2(\Omega)}.
\end{equation*}
}
{
The estimate
\begin{equation*}
\rho^+ \| w\|_{H^1(\Omega^+)} \le C \|f\|_{L^2(\Omega)}
\end{equation*}
follows from Poincare's inequality and  Proposition \ref{prop3}.
Hence, we have
\begin{equation*}
\rho^+ \|w\|_{H^2(\Omega^+)} \le C \, \|f\|_{L^2(\Omega)}.
\end{equation*}
The result now follows after noting that $\|D^2 u\|_{L^2(\Omega^+)} \le  \|w\|_{H^2(\Omega^+)}$. }
\end{proof}

Combining Theorem \ref{maintheorem} and the previous propositions we have the following corollary.
\begin{corollary} \label{coro1}
Let $u$ be the solution to \eqref{Problem} and $u_h \in V_h$ be its finite element approximation. If $\Omega$ is convex, we have
\begin{equation}
\|u-u_h\|_V  \,\,\le\,\, \frac{C \, h {(1+\kappa)}}{\sqrt{\rho^-}} \|f\|_{L^2(\Omega)}.
\end{equation}
{The constant $C$ depends on the constant $\Ce$ from \eqref{extension} and the constant $\Cr$ from \eqref{H2regularity}}.
\end{corollary}

\subsection{$L^2$-Error Estimates}\label{L2estimate}

We now prove an $L^2$ error estimate using a duality argument.
\begin{theorem} \label{coro2} Let $u$ be the solution to problem \eqref{Problem} and $u_h \in V_h$ be its
finite element approximation solution to \eqref{fem}. Assume that $\Omega$ is convex. Then, there exists a constant $C>0$ independent of $h$, $\rho^-$ and $\rho+$, such that
\begin{equation}\label{mainerror3}
\|u-u_h\|_{L^2(\Omega)} \,\, \le\,\,  \frac{C\, h^2{(1+\kappa)^2} }{\rho^-} \|f\|_{L^2(\Omega)}.
\end{equation}
{The constant $C$ depends on $\Ce$ from \eqref{extension} and $\Cr$ from \eqref{H2regularity}\revblue{.}}
\end{theorem}
\begin{proof}
Let $\phi$ be a solution of the problem
\begin{subequations}\label{Problemphi}
\begin{alignat}{3}
    -\rho^{\pm} \Delta \phi^{\pm} &= (u-u_h)^{\pm} \qquad &\mbox{in } &\Omega^{\pm}, \label{Problemphi:a} \\
                                           \phi &=0            &\mbox{on }&\partial \Omega, \label{Problemphi:b}\\
                            \left[\phi\right]&=0            &\mbox{on }&\Gamma, \label{Problemphi:c} \\
  \left[\rho D_{\n{n}}  \phi  \right]&=0            &\mbox{on }&\Gamma. \label{Problemphi:d}
\end{alignat}
\end{subequations}
We have
\begin{equation*}
\|u-u_h\|_{L^2(\Omega)}^2=\int_{\Omega}( u-u_h) u-\int_{\Omega}(u-u_h) u_h=a_h(\phi, u)-a_h(\phi_h, u_h),
\end{equation*}
where $\phi_h$ is the finite element approximation of $\phi$. Therefore, we see that
\begin{alignat}{1}
\|u-u_h\|_{L^2(\Omega)}^2\,\,=&\,\,a_h(\phi, u-u_h)+a_h(\phi-\phi_h,u_h) \label{u-uhL2} \\
\,\,=&\,\,a_h(\phi-\phi_h, u-u_h)+a_h(u-u_h, \phi_h)+a_h(\phi-\phi_h,u_h)  \nonumber \\
\,\,=& \,\,a_h(\phi-\phi_h, u-u_h)+a_h(u-u_h, \phi_h-I_h \phi)+a_h(u-u_h, I_h \phi) \nonumber  \\
\,\,  & \,\,+a_h(\phi-\phi_h,u_h-I_h u) +a_h(\phi-\phi_h,I_h u).\nonumber
\end{alignat}
Using continuity of the bilinear form we have
\begin{equation*}
a_h(\phi-\phi_h, u-u_h) \,\,\le\,\, C \|\phi-\phi_h\|_W \, \|u-u_h\|_W.
\end{equation*}
Using the triangle inequality we have
\begin{equation*}
\|u-u_h\|_W \,\,\le\,\, \|u-I_h u\|_W+\|I_h u-u_h\|_W.
\end{equation*}
It is not difficult to show that
\begin{equation*}
\|I_h u-u_h\|_W\,\, \le\,\, C \|I_h u-u_h\|_V.
\end{equation*}
Hence,
\begin{equation*}
\|u-u_h\|_W \,\,\le\,\, \|u-I_h u\|_W+\| u-u_h\|_V.
\end{equation*}
Now using Theorem \ref{normestimate}, Theorem \ref{maintheorem}, \eqref{H2regularity} and \eqref{energyinequality}  we get
\begin{equation*}
\|u-u_h\|_W \,\,\le\,\, \frac{C\, h {(1+\kappa)}}{\sqrt{\rho^-}} \|f\|_{L^2(\Omega)}.
\end{equation*}
Similarly,
\begin{equation*}
\|\phi-\phi_h\|_W \,\,\le\,\, \frac{C\, h {(1+\kappa)}}{\sqrt{\rho^-}} \|u-u_h\|_{L^2(\Omega)},
\end{equation*}
and hence we have a bound for the first term in \eqref{u-uhL2}
\begin{equation*}
a_h(\phi-\phi_h, u-u_h) \,\, \le\,\,  \frac{C\, h^2 {(1+\kappa)^2}}{\rho^-} \|f\|_{L^2(\Omega)} \,\|u-u_h\|_{L^2(\Omega)}.
\end{equation*}
Using  Lemma \ref{inconsistency}, \eqref{H2regularity} and \eqref{energyinequality} we have
\begin{equation*}
a_h(u-u_h, \phi_h-I_h \phi) \,\,\le\,\,  \frac{C\, h {(1+\kappa)}}{\sqrt{\rho^-}} \|f\|_{L^2(\Omega)}  \|I_h \phi-\phi_h\|_V,
\end{equation*}
which implies the following bound for the second term in \eqref{u-uhL2}
\begin{equation*}
a_h(u-u_h, \phi_h-I_h \phi) \,\,\le\,\,  \frac{C\, h^2 {(1+\kappa)^2}}{\rho^-} \|f\|_{L^2(\Omega)} \|u-u_h\|_{L^2(\Omega)}.
\end{equation*}
Analogously, for the fourth term in \eqref{u-uhL2}we have
\begin{equation*}
a_h(\phi-\phi_h,u_h-I_h u) \,\,\le\,\,  \frac{C\, h^2 {(1+\kappa)^2}}{\rho^-} \|f\|_{L^2(\Omega)} \|u-u_h\|_{L^2(\Omega)}.
\end{equation*}

For the third term in \eqref{u-uhL2} we have
\begin{equation}
a_h(u-u_h, I_h \phi) \,\,= \,\,\sum_{T \in \mathcal{T}_h^{\Gamma}}  \int_{\TG} (\rho^- D_{\n{n}^-} u^-) [I_h \phi] .
\end{equation}
Using the Cauchy-Schwarz inequality we obtain
\begin{equation*}\label{3rdtermexpand}
a_h(u-u_h, I_h \phi)\,\, \leq\,\,  \|\sqrt{\rho^-} D_{\n{n}^-} u^-\|_{L^2(\Gamma)}  \sqrt{\rho^-} \left(\sum_{T \in \mathcal{T}_h^{\Gamma}}  \|[I_h \phi]\|_{L^2(\TG)}^2 \right)^{1/2},
\end{equation*}
For the first term in \eqref{3rdtermexpand} we apply a trace inequality to obtain
\begin{equation*}
  \|\sqrt{\rho^-} D_{\n{n}} u^-\|_{L^2(\Gamma)} \,\, \le\,\,  C \sqrt{\rho^-}  (\|D^2 u\|_{L^2(\Omega^-)}+\|D u\|_{L^2(\Omega^-)}) \le \frac{C}{\sqrt{\rho^-}} \|f\|_{L^2(\Omega)}.
\end{equation*}
where we used \eqref{H2regularity}, \eqref{energyinequality} and the fact that $\rho^- \le \rho^+$. Now for the second term in \eqref{3rdtermexpand}, we note that $[I_h \phi]=0$ on $L_T$ and that $L_T$ is at most distance $\mathcal{O}(s^2)$ from $\TG$ where $s=|\TG|$. Therefore we can use, Taylor's theorem to show that
\begin{equation*}
[I_h \phi](x) \,\,\le\,\, C {\|\n{X}''\|_{L^{\infty}}}s^2 |D [I_h \phi]|(x) \quad \forall x \in \TG.
\end{equation*}
Hence,
\begin{equation*}
 {\|[I_h \phi]\|_{L^2(\TG)} \,\,\le\,\, C {\kappa} r^2\|D [I_h \phi]\|_{L^2(\TG)} \le C {\kappa} h_T^2  \|D [I_h \phi]\|_{L^2(\TG)},}
\end{equation*}
where we used that $s \le C \,h_T$.  Consequently, adding and subtracting $D [\phi]$
\begin{equation*}
\left(\sum_{T \in \mathcal{T}_h^{\Gamma}} \|[I_h \phi]\|_{L^2(\TG)}^2 \right)^{1/2}\,\, \le\,\, C {\kappa}  \left(\sum_{T \in \mathcal{T}_h^{\Gamma}}  h_T^4 \|D [I_h \phi-\phi]\|_{L^2(\TG)}^2 \right)^{1/2}+C {\kappa}  \left(\sum_{T \in \mathcal{T}_h^{\Gamma}} {h_T^4} \|D [\phi]\|_{L^2(\TG)}^2 \right)^{1/2}.
\end{equation*}
Observe that
\begin{equation*}h^2 \|D [\phi]\|_{L^2(\Gamma)} \le  C  h^2 (\|D \phi^+\|_{L^2(\Gamma)} +  \|D \phi^-\|_{L^2(\Gamma)} ).
\end{equation*}
Application of trace inequality gives
\begin{equation*}
 \left(\sum_{T \in \mathcal{T}_h^{\Gamma}} \|D [\phi]\|_{L^2(\TG)}^2 \right)^{1/2} \le C  h^2 \left(\|D^2 \phi\|_{L^2(\Omega^-)}+ \|D \phi\|_{L^2(\Omega^-)} + \|D^2 \phi\|_{L^2(\Omega^+)}+ \|D \phi\|_{L^2(\Omega^+)} \right).
\end{equation*}
For the other term we use that $I_h \phi|_T=I_T \phi$ and use a trace inequality  to bound
\begin{alignat*}{1}
 \|D [I_h \phi-\phi]\|_{L^2(\TG)}\,\,=\,\,&  \|D [I_T \phi-\phi]\|_{L^2(\TG)}\\
  \le \,\,& \frac{C }{\sqrt{h_T}} (  \| D (I_T \phi-\phi)\|_{L^2(\wT^{-})}+ \| D (I_T \phi-\phi)\|_{L^2(\wT^{+})})\\
                                                                         & \,\,  + C \sqrt{h_T} (\| D^2 \phi\|_{L^2(\wT^-)}+ \| D^2 \phi\|_{L^2(\wT^+)}).
\end{alignat*}
From \eqref{localcurve1} and \eqref{localcurve2} (using that $\rho^- \le \rho^+$) we obtain
\begin{alignat*}{1}
\| D (I_T \phi-\phi)\|_{L^2(\wT^{-})}+ \| D (I_T \phi-\phi)\|_{L^2(\wT^{+})} \,\,\le\,\, &  C h_T (\|D^2 \phi\|_{L^2(\wT^{-})}+ \|D^2 \phi\|_{L^2(\wT^{+})}).
\end{alignat*}
Therefore,
\begin{alignat*}{1}
\left(\sum_{T \in \mathcal{T}_h^{\Gamma}} h_T^4 \|D [I_h \phi-\phi]\|_{L^2(\TG)}^2 \right)^{1/2} \,\,\le\,\, & C\, h^2 {\kappa}( \|D^2 \phi\|_{L^2(\Omega^-)}+ \|D \phi\|_{L^2(\Omega^-)}) \\
& +C\, h^2 (\|D^2 \phi\|_{L^2(\Omega^+)}+ \|D \phi\|_{L^2(\Omega^+)}).
\end{alignat*}
Hence, using the regularity result \eqref{H2regularity} and \eqref{energyinequality} we have
\begin{equation*}
 \sqrt{\rho^-} \left(\sum_{T \in \mathcal{T}_h^{\Gamma}}  \|[I_h \phi]\|_{L^2(\TG)}^2 \right)^{1/2} \,\,\le\,\, \frac{C h^2 {\kappa}}{\sqrt{\rho^-}} \|u-u_h\|_{L^2(\Omega)}.
\end{equation*}
Thus, we obtain the bound for the third term in \eqref{u-uhL2}
\begin{equation*}
a_h(u-u_h, I_h \phi) \,\,\le\,\, \frac{C h^2{\kappa} }{\rho^-} \|f\|_{L^2(\Omega)} \|u-u_h\|_{L^2(\Omega)}.
\end{equation*}
In a similar fashion we can prove
\begin{equation*}
a_h(\phi-\phi_h, I_h u) \,\,\leq\,\,\frac{C h^2 {\kappa} }{\rho^-} \|f\|_{L^2(\Omega)} \|u-u_h\|_{L^2(\Omega)}.
\end{equation*}
The proof  is complete after combining the above inequalities for the terms in \eqref{u-uhL2}.
\end{proof}

\section{Extensions and final remarks}\label{extensionsrelated}

\subsection{Extension to three dimensions}

We now show that the space $S^1(T)$, introduced in (\ref{localFEspace}),
can easily be  extended in three-dimensions. We consider the problem \eqref{Problem} where $\Omega$ is a three-dimensional domain  and $\Gamma$ is a {simple, closed,} $C^2$ surface. Let now $\mathcal{T}_h$  be a simplicial triangulation of $\Omega$. Let $\mathcal{E}^h$ be all the faces of the mesh $\mathcal{T}_h$. Finally, let $\EhG$ be the set of all faces that belong to a tetrahedron that intersects $\Gamma$.

We now define the local finite element space. Let $T \in \mathcal{T}_h$ be a tetrahedron.  Let $x_0$ be a fixed point on $\Gamma \cap T$.   Let $\n{n}_0^+$ be the outward pointing unit vector normal to $\Omega^+$ at $x_0$.  Let $\n{t}_0^+, \n{s}_0^+$ be such that $\{\n{n}_0^+, \n{t}_0^+, \n{s}_0^+\}$ forms an orthonormal system.

Given  $v \in \mathbb{P}^1(T^+)$ there exists a unique  $\C(v)
\in \mathbb{P}^1(T^-)$ satisfying
\begin{alignat*}{1}
\C(v)(x_0)&\,\,=\,\ v(x_0), \label{eqE} \\
  \left( D_{\n{t}_0^+} \C(v) \right)(x_0) &\,\,=\,\,
\left( D_{\n{t}_0^+} v \right)(x_0),   \\
\left( D_{\n{s}_0^+} \C(v) \right)(x_0) &\,\,=\,\,
\left( D_{\n{s}_0^+} v \right)(x_0),   \\
\rho^-\left( D_{\n{n}_0^+} \C(v) \right)(x_0) &\,\,=\,\,
\rho^+\left( D_{\n{n}_0^+} v \right)(x_0).
\end{alignat*}

Given $T\in \mathcal{T}_h^\Gamma$ and for each $v
\in \mathbb{P}^1(T^+)$ we can consider the unique corresponding function
\[
G(v)=
\begin{cases}
v & \text{ in } T^{+}, \\
\C(v) & \text{ in } T^{-}.
\end{cases} \]

Let $\text{ span } \{v_1, v_2, v_3, v_4\}$ be a basis for $\mathbb{P}^1(T)$
restricted to $T^+$.  Then we define the local finite element space
\begin{equation*}
S^1(T)\,\,= \,\,
\begin{cases}
\text{span } \big\{G(v_1), G(v_2), G(v_3), G(v_4) \big\}&, \mbox{ if } T \in \mathcal{T}_h^\Gamma \\
\qquad\qquad\qquad\mathbb{P}^1(T) &, \mbox{ if } T \in \mathcal{T}_h \backslash \mathcal{T}_h^\Gamma.
\end{cases}
\end{equation*}

Consequently, the global finite element space is given by
\begin{equation*}
V_h\,:=\, \left\{ v\,:\, v|_T \in S^1(T),\, \forall T \in \mathcal{T}_h,\, v \text{ is continuous across all faces in } \EhGn \right\}.
\end{equation*}

\revblue{We believe that \eqref{fem} can be extend and analyzed to the three-dimensional case,
where now $e^\pm = e \cap \Omega^\pm$ are  pieces of a face $e$ of a tetrahedra, and considering appropriate stabilization parameters. This could be a subject for a future work.}

\subsection{Alternative Local Spaces}

An alternative approach to enforce weak continuity of our local space $S^1(T)$ is normally used in the literature (see \cite{MR2684351}). Instead of enforcing continuity of function and tangential derivative at $x_0$ one imposes continuity at two distinct points   $x_1, x_2$. More precisely, define $x_1, x_2$ be the two points in $\Gamma$ that intersect $\partial T$. Then, one can define  $\C(v)$ (in contrast to the definition in Lemma \ref{Lemma1}) by
\begin{alignat}{1}
\C(v)(x_i)&\,\,=\,\,v(x_i)\qquad \mbox{for }\, i\,=\,1,2 \nonumber \\
\rho^- ( D_{\n{n}^+_0} \C(v))(x_0) &\,\,=\,\,  \rho^+ ( D_{\n{n}^+_0} v)(x_0). \nonumber
\end{alignat}
Subsequently, $I_Tu$ is now defined by
\begin{eqnarray*}
(I^-_T u)(x_i) \,\,:=&  (J_T u^+_E)(x_i) &=:\,\,  (I^+_T u) (x_i) \qquad \mbox{for }\, i\,=\,1,2 \\
\rho^- (D_{\n{n}^+_0} I^-_T u)(x_0) \,\,:=& \rho^+ (D_{\n{n}^+_0} J_T u^+_E)(x_0) &=:\,\,
\rho^+ (D_{\n{n}^+_0} I^+_T u)(x_0).
\end{eqnarray*}
Defining the finite element method using these local spaces, we can prove all the a-priori estimates above.


Another alternative is to enforce the matching conditions by averaging on $T_{\Gamma}$.
More precisely,  one can redefine $\C(v)$ in Lemma \ref{Lemma1} by
\begin{alignat}{1}
\int_{\TG} \C(v) \,ds &\,\,:=\,\ \int_{\TG} v, \nonumber  \\
\int_{\TG} \C(v)\,s\, \,ds &\,\,:=\,\ \int_{\TG} v\,s, \nonumber \\
\int_{\TG} \rho^+D_{\n{n}^+} \C(v)  &\,\,:=\,\,
\int_{\TG} \rho^-D_{\n{n}^+}  v, \nonumber
\end{alignat}
or alternatively, one can replace the second equation by
\[
\int_{\TG} D_{\n{t}^+} \C(v)\, ds \,\,:=\,\,
\int_{\TG} D_{\n{t}^+}  \,v ds.
\]
Even though their numerical implementation is more complicated
and require numerical integrations, these alternatives have the
advantage that the analysis is shorter especially for the consistency
error, and also can be formulated using Lagrange multipliers.

\subsection{Extension to {C}artesian grids}  We note
that the analysis and methods developed here
can be easily extended to Cartesian grids and $\mathbb{Q}^1$ elements.
One possibility would be to use same local spaces $S^1(T)$ defined
before for quadrangular elements $T  \in \mathcal{T}_h^\Gamma$, that is,
one has three degrees of freedom for $T  \in \mathcal{T}_h^\Gamma$ and
four  degrees otherwise. In case one wants to have four degrees of
freedom for $T  \in \mathcal{T}_h^\Gamma$, let the space $S^1(T)$ then be
a piecewise bilinear function under the coordinates $\n{n}_0, \n{t}_0$
and impose
\begin{alignat}{1}
\C(v)(x_0) &\,\,:=\,\, v(x_0) \nonumber \\
  ( D_{\n{t}^+_0} \C(v) )(x_0) &\,\,:=\,\,  ( D_{\n{t}^+_0} v )(x_0) \nonumber  \\
\rho^-( D_{\n{n}^+_0} \C(v) )(x_0) &\,\,:=\,\,
\rho^+( D_{\n{n}^+_0} v )(x_0) \nonumber\\
\rho^-( D_{\n{n}^+_0} D_{\n{t}^+_0} \C(v) )(x_0) &\,\,:=\,\,
\rho^+( D_{\n{n}^+_0} D_{\n{t}^+_0} v )(x_0) \nonumber
\end{alignat}

\subsection{Alternative Bilinear forms}
In this section we discuss alternative bilinear forms. We will point out what are the {theoretical} difficulties in analyzing these alternative methods. When we provide numerical experiments in the following section we will also see that, although we cannot prove stability for  some methods, they sometimes do well experimentally {in some of the norms}.

Firstly, formulation without flux stabilization has been used for example in \cite{MR3338673}. Methods \eqref{a_h2} and \eqref{a_h4} experiment this direction.
\begin{alignat}{1}
a_h(w, v)\,\,:=\,\,& \int_{\Omega} \rho \nabla_h w \cdot \nabla_h v - \sum_{e \in \EhG}  \int_e  \left(\big\{\rho \nabla_h v \big\}
\cdot \llbracket w \rrbracket+   \big\{\rho \nabla_h  w \big\} \cdot \llbracket v \rrbracket\right) \label{a_h2}\\
& +  \gamma\,\sum_{e \in \EhG} \frac{1}{|e|} \left( \int_{e^-} \rho^- \llbracket w \rrbracket \cdot \llbracket v \rrbracket+  \int_{e^+} \rho^+ \llbracket w \rrbracket\cdot \llbracket v \rrbracket \right). \nonumber
\end{alignat}

\begin{alignat}{1}
a_h(w, v)\,\,:=\,\,& \int_{\Omega} \rho \nabla_h w \cdot \nabla_h v - \sum_{e \in \EhG}  \int_e  \left(\big\{\rho \nabla_h v \big\}
\cdot \llbracket w \rrbracket+   \big\{\rho \nabla_h  w \big\} \cdot \llbracket v \rrbracket\right) \label{a_h4}\\
& +  \gamma\,\sum_{e \in \EhG} \left( \frac{1}{|e^-|}  \int_{e^-} \rho^- \llbracket w \rrbracket \cdot \llbracket v \rrbracket+ \frac{1}{|e^+|}   \int_{e^+} \rho^+ \llbracket w \rrbracket\cdot \llbracket v \rrbracket \right). \nonumber
\end{alignat}
The problem with these methods is that we cannot { establish neither
the optimal inconsistency error nor the coercivity independent of the contrast
and the mesh.}

{ We note that methods \eqref{a_h2} and \eqref{a_h4} seem to do well
numerically in the $L^2$ and energy norms, however, they are
mesh dependent in the $L^\infty$ norm and they do not converge
in the weighted $W^{1,\infty}$ norm.}

The method that seems to do the best numerically, slightly better than the proposed method, is the following one:
\begin{alignat}{1}
a_h(w, v)\,\,:=\,\,& \int_{\Omega} \rho \nabla_h w \cdot \nabla_h v - \sum_{e \in \EhG}  \int_e  \left(\big\{\rho \nabla_h v \big\}
\cdot \llbracket w \rrbracket+   \big\{\rho \nabla_h  w \big\} \cdot \llbracket v \rrbracket\right) \label{a_h5}\\
& +  \gamma\,\sum_{e \in \EhG} \left(\frac{1}{|e^-|} \int_{e^-} \rho^- \llbracket w \rrbracket \cdot \llbracket v \rrbracket+ \frac{1}{|e^+|} \int_{e^+} \rho^+ \llbracket w \rrbracket\cdot \llbracket v \rrbracket \right) \nonumber   \\
&  \gamma_F\,\sum_{e \in \EhG}{|e|}  \left(\int_{e^-} \rho^-\llbracket \nabla_h  v\rrbracket \, \llbracket \nabla_h  w \rrbracket + \int_{e^+} \rho^+\llbracket \nabla_h  v\rrbracket\,\llbracket \nabla_h  w\rrbracket \right).\nonumber
\end{alignat}
The difference between this method and the one analyzed in this paper is that we are using a stronger flux stabilization (i.e. we replace $|e^\pm|$ by $|e|$ and we introduce a  flux stabilization parameter $\gamma_F$).  It is not difficult to see, that we can prove all the error estimates contained in this paper for this method. In particular, the coercivity of the bilinear form is obvious since we are adding even more stabilization. Clearly, now we would have to redefine our $V$ and $W$ norms but the approximation properties will still hold. In summary, Theorem \ref{maintheorem} and
Corollaries \ref{coro1} and \ref{coro2} hold for this formulation.

{ A very natural question would be if we can penalize the jump terms with $1/|e|$ instead of $1/|e^{\pm}|$ and get a stable and optimally convergent method independent of contrast. The answer is yes, as long as we penalize the jump of the full gradient instead of the flux.
\begin{alignat}{1}
a_h(w, v)\,\,:=\,\,& \int_{\Omega} \rho \nabla_h w \cdot \nabla_h v - \sum_{e \in \EhG}  \int_e  \left(\big\{\rho \nabla_h v \big\}
\cdot \llbracket w \rrbracket+   \big\{\rho \nabla_h  w \big\} \cdot \llbracket v \rrbracket\right) \label{a_h3}\\
& +  \gamma\,\sum_{e \in \EhG} \frac{1}{|e|}\left( \int_{e^-} \rho^- \llbracket w \rrbracket \cdot \llbracket v \rrbracket+  \int_{e^+} \rho^+ \llbracket w \rrbracket\cdot \llbracket v \rrbracket \right) \nonumber   \\
&  \gamma_F\,\sum_{e \in \EhG} {|e|}\left( \int_{e^-} \rho^- \left[\nabla    v\right]_{\otimes}    \cdot  \left[\nabla w\right]_{\otimes} + \int_{e^+} \rho^+  \left[\nabla    v\right]_{\otimes}   \cdot  \left[\nabla w\right]_{\otimes}       \right),\nonumber
\end{alignat}
where the jumps in the last line are defined by
\[
\left[\nabla    v\right]_{\otimes}  = \nabla v^{-} \otimes \n{n}^{-} + \nabla v^{+} \otimes \n{n}^{+}.
\]
and the symbol  $\otimes$ denotes the outer product between vectors.
}
{
We note that in the coercivity and inconsistency
error analysis proofs, we have used the inverse inequality
$ |e^-| \,\|D_{\n{n}} v\|_{L^2(e^-)} \leq C\, \|\nabla v\|_{L^2(T^-)}$ (where $T^-$ is a triangle with edge $e$, and similar result for $e^+$) in particular in the estimation of
 the left-hand side of (\ref{Dn1bound}). For method \eqref{a_h3} we want to avoid using this inverse estimate and the
factor $1/|e^-|$ (also $1/|e^+|$). Let $x$ be the vertex of $e$ belonging to $\Omega^-$.
Using the ideas in the proof of Lemma \ref{twotri},  we can find a sequence of elements
$\{T_2, T_3,\cdots, T_{N}\}$  in the patch of $x$,  and common edges
$\{e_2, e_3, \cdots, e_{N-1}\}$ with $\bar{e}_i  =
\overline{T}_i \cap \overline{T}_{i+1}$ and $e_i^- := e_i \cap \Omega^-$,
such that, $|e^-|\leq {\sigma} |e_2^-| \leq {\sigma}^2 |e_3^-| \leq {\sigma}^{N-2} |e_{N-1}^-| = {\sigma}^{N-1} |e|$ where ${\sigma}$ a positive constant bounded uniformly from below by zero  and $N$ is bounded depending on the shape regularity of the mesh.
Then, we bound the left-hand side of  (\ref{Dn1bound}) by using recursively
\begin{equation*}
\|\nabla v|_{T_{i-1}}\|^2_{L^2(e_i^-)} \leq 2 \left(
\|   \left[\nabla v  \right]_{\otimes} \|^2_{L^2(e_i^-)} +
\|\nabla v|_{T_{i}}\|^2_{L^2(e_i^-)} \right),
\end{equation*}
and
\begin{equation*}
\|\nabla v|_{T_i}\|^2_{L^2(e_i^-)} \leq
1/{\sigma} \|\nabla v|_{T_i}\|^2_{L^2(e_{i+1}^-)},
\end{equation*}
and then use that
\begin{equation*}
\|\nabla v|_{T_{N}}\|^2_{L^2(e_{N-1}^-)}
\leq( C/|e|)\, \|\nabla v\|^2_{L^2(T_N^-)}.
\end{equation*}
Hence,
\begin{equation*}
|e| \, \|\nabla v\|_{L^2(e^-)}^2 \le C \,  \sum_{i=2}^N  |e| \, \|  \left[\nabla v  \right]_{\otimes} \|^2_{L^2(e_i^-)}+ \frac{ C}{|e|} \, \|\nabla v\|^2_{L^2(T_N^-)}.
\end{equation*}
}

\subsection{Final remarks}\label{finalremarks}
In this article  rigorous error estimates independent of contrast were {derived}. However, many interesting research questions remain. Firstly, we kept track, as much as possible, of how the constants depend on the geometry (e.g. curvature). However, there are two constants, $\Ce$ and $\Cr$, that we did not investigate in detail how they scale with geometric quantities such as maximum curvature and the radius $r$ of the tubular neighborhood. We believe it is possible to prove a bound for $\Ce$ in terms of geometric quantities, but for the sake of simplicity we did not investigate it here.  In addition, the regularity constant $\Cr$ appearing in \ref{H2regularity} depends on the geometry.  Addressing these issues will lead to error estimates that are completely explicit on their dependence on the curvature and radius $r$ of the tubular neighborhood.

Secondly, our method and error estimates do not consider problems with high curvature. An interesting line of research would be to investigate problems with interfaces of arbitrary curvature. In this direction, perhaps a combination of the method presented here and the multicale method by Chu et al. in \cite{MR2684351} could be a possible approach.

Thirdly, rigorous error analysis for higher order Immersed Finite Element methods ($k>1$) remains open. It would be interesting to study this, in particular, for high-contrast problems.

Finally, we would like to study the sharpness of our Assumptions (1)-(3). For instance, we observe that our method is still well-defined without Assumption (3).
\section{Numerical examples}\label{Numericssection}

In this section we explore the properties of the methods presented in sections above applied to the two dimensional interface problem \eqref{Problem}. In particular, we are interested in the computation of the following errors and their respective \revblue{estimated order} of convergence
\begin{align*}
e_h^0\,\,& :=\,\, \|u -  u_h\|_{L^2(\Omega)},&  e_h^\infty\,\,& :=\,\, \|u - u_h\|_{L^\infty(\Omega)},   \\
e_h^1\,\,& :=\,\, \|\sqrt{\rho}\nabla_h( u -  u_h)\|_{L^2(\Omega)},&
e_h^{1,\infty}\,\,& :=\,\, \|\sqrt{\rho}\nabla_h( u -  u_h)\|_{L^\infty(\Omega)} ,\\
\bar{e}_h^1\,\,& :=\,\, \|\rho\nabla_h (u -  u_h)\|_{L^2(\Omega)},&
\bar{e}_h^{1,\infty}\,\,& :=\,\, \|\rho\nabla_h (u -  u_h)\|_{L^\infty(\Omega)} ,
\\
e_h^{\n{n},\infty}\,\,& :=\,\, \|\rho D_{\n{n}} (u- u_h)\|_{L^\infty(\Gamma)} ,& {\tilde{e}_h^{1,\infty}}\,\,& :=\,\, {\|\rho\nabla_h( u - u_h)\|_{L^\infty(\Omega\backslash(\cup\{T: T\in\mathcal{T}_h^\Gamma\}))}}
\end{align*}
\[ \revblue{\mathrm{e.o.c.}} \,\,:=\,\, \frac{\log(e_{h_{l+1}}/e_{h_l})}{ \log(h_{l+1}/h_{l})}.\]
{Computation (or approximation) of the $L^\infty$ norms are performed evaluating the error at the following set of points: if an element does not intersect the interface then the set of points are the nodes and the centroid of the element, if the element does intersect the interface we evaluate at the nodes and at the two points intersections of the interface with the edges of the element.}
We observe that the errors corresponding to the $L^2$ norm and $W^{1,2}$ weighted with $\sqrt{\rho}$ semi-norm correspond to the only results proved in this paper, Theorem \ref{coro2} and \ref{maintheorem}  respectively. We expect optimal convergence, second order in the $L^2$ norm and first order in the weighted $W^{1,2}$ semi-norm.
We also compute the analogous in the $L^\infty$ norm and $W^{1,\infty}$ weighted with $\sqrt{\rho}$ semi-norm. In addition, we compute the errors in the $W^{1,2}$ and $W^{1,\infty}$ semi-norm weighted with  $\rho$.  Note that the estimate for the interpolation error is also achieved in this semi-norm. Indeed, {this} is a consequence of Lemma \ref{localcurve} and Proposition \ref{prop1} and \ref{prop3}
\begin{align*}
\|\rho \nabla_h (u-I_h u)\|_{L^2(\Omega)}\,\,\leq&\,\, C \|f\|_{L^2(\Omega)}.
\end{align*}

A revealing result of our experiments is that the ratio of convergence of the error is optimal when we compute using the triangles non intersecting the interface.  Error $\tilde{e}_h^{1,\infty}$ illustrates this observation. Finally, error $e_h^{\n{n},\infty}$ is a standard error for interface problems, illustrating the approximation of the normal derivative on the interface.

The experiment presented below shows that method  \eqref{fem}-\eqref{a_h5} produces the best results. This method (and all the others) does not seem optimal for the $W^{1,\infty}$ error weighted with $\rho$. However, it is optimal when we do not consider the elements in $\mathcal{T}_h^\Gamma$. We highlight that this result was not remotely addressed in this paper, and this kind of estimates appear to be more difficult. For the method that we analyzed  \eqref{fem}-\eqref{a_h} we observe an uncertain behavior in the error $e_h^{1,\infty}$ for the last mesh, not achieving the optimal convergence as in the previous method.
We also present tables for one of the method without flux stabilization, method \eqref{fem}-\eqref{a_h4}. We observe non convergence so far for the error $e_h^{1,\infty}$ and in the last mesh and a deterioration in the $L^\infty$ norm. We also do not have the optimal convergence for the normal flux error $e_h^{\n{n},\infty}$.

In our numerical experiment we consider the two dimensional domain $\Omega = (-1,1)^2$ with the  immersed interface $\Gamma = \{ \n{x} \in\Omega : \n{x}_1^2+\n{x}_2^2 = R^2\}$.  We define  $\Omega^- :=\{ \n{x} \in\Omega : \n{x}_1^2+\n{x}_2^2 < R^2\}$ and $\Omega^+ = \Omega\backslash(\Omega^-\cup \Gamma)$. Example \ref{Ex1} considers the following exact solution
\begin{equation*}
    u(\n{x})\,\,=\,\, \left\{
                        \begin{array}{ll}
                          \frac{\revblue{R}^\alpha}{\rho^-} &, \hbox{ if } \n{x}\in \Omega^-, \\
                          \frac{\revblue{R}^\alpha}{\rho^+} +\frac{1}{3^\alpha}(\frac{1}{\rho^-} - \frac{1}{\rho^+}) &,  \hbox{ if } \n{x}\in \Omega^+, \
                        \end{array}
                      \right.
    \end{equation*}
where $\revblue{R} = \sqrt{\n{x}_1^2+\n{x}_2^2}$ and $\alpha = 2$. For example \ref{Ex1} we have $\Gamma = \partial \Omega^-$. Similar results were obtained for the case  $\Gamma = \partial \Omega^+$. We provide plots of the approximate solution for both cases.

Finite element uniform triangular meshes $\mathcal{T}_h$ non matching the interface were used. In the tables we compute with $h = 2^{-(l+3/2)}$, for $l = 1,\,...,\,7$.

All the computations were performed in MATLAB, including solution of the linear system by means of ``$\backslash$''.

\begin{enumerate}[label=\bfseries (\arabic*)]
\item\label{Ex1} Case $\Gamma = \partial \Omega^-$, $\rho^+ = 10^4$ and $\rho^-=1$.  Tables \ref{Table:Ex1_1} and \ref{Table:Ex1_2} show the results obtained by method \eqref{fem}-\eqref{a_h5} with stabilization parameters $\gamma = 10$ and $\gamma_F=10$. Note that this method has a stronger flux stabilization that the method analyzed in the paper.

\begin{table}[!htp]
\small
\ra{1.1}
\begin{center}
\begin{tabular}{@{}l@{\hskip .4in}c@{\hskip .2in}c@{\hskip .4in}c@{\hskip .2in}c@{\hskip .4in}c@{\hskip .2in}c@{\hskip .4in}c@{\hskip .2in}c@{}}\toprule
$l$ & $e_h^0$ & \revblue{e.o.c.} & $e_h^\infty$ & \revblue{e.o.c.}  & $e_h^1$ & \revblue{e.o.c.} & $e_h^{1,\infty}$ & \revblue{e.o.c.} \\ \midrule
1  &   8.2e-3   &             &    2.5e-2    &              &      1.1e-1    &                &      3.7e-1     &\\
2  &   1.7e-3   &   2.28   &    5.7e-3    &   2.12    &      4.4e-2    &     1.30    &      2.1e-1     &      0.86\\
3  &   2.7e-4   &   2.63   &    1.3e-3    &   2.19    &      1.8e-2    &     1.29    &      9.7e-2     &      1.08\\
4  &   4.6e-5   &   2.57   &    3.2e-4    &   1.97    &      8.3e-3    &     1.12    &      5.2e-2     &      0.90\\
5  &   9.0e-6   &   2.34   &    7.2e-5    &   2.15    &      3.9e-3    &     1.07    &      2.5e-2     &      1.08\\
6  &   2.0e-6   &   2.19   &    1.8e-5    &   2.01    &      1.9e-3    &     1.03    &      1.3e-2     &      0.92\\
7  &   4.7e-7   &   2.08   &    4.6e-6    &   1.94    &      9.5e-4    &     1.02    &      6.8e-3     &      0.94\\
\bottomrule
\end{tabular}\vskip3mm
\begin{tabular}{@{}l@{\hskip .4in}c@{\hskip .2in}c@{\hskip .4in}c@{\hskip .2in}c@{\hskip .4in}c@{\hskip .2in}c@{\hskip .4in}c@{\hskip .2in}c@{}}\toprule
$l$ & $\bar{e}_h^1$ & \revblue{e.o.c.} & $\bar{e}_h^{1,\infty}$ & \revblue{e.o.c.} & $\tilde{e}_h^{1,\infty}$ & \revblue{e.o.c.} & $e_h^{\n{n},\infty}$ & \revblue{e.o.c.}    \\ \midrule
1  &   3.9e-1    &               &     7.0e-1     &                &  7.0e-1     &               & 3.7e-1    &            \\
2  &   1.6e-1    &    1.32    &     6.1e-1     &     0.19    &  4.0e-1     &     0.81   & 2.1e-1    &     0.86\\
3  &   6.4e-2    &    1.29    &     1.9e-1     &     1.68    &  1.9e-1     &     1.07   & 9.7e-2    &     1.08\\
4  &   2.9e-2    &    1.15    &     2.2e-1     &     -0.20   &  1.0e-1     &     0.91   & 4.9e-2    &     1.00\\
5  &   1.4e-2    &    1.09    &     1.4e-1     &     0.65    &   5.0e-2    &     1.01   & 2.5e-2    &     0.98\\
6  &   6.6e-3    &    1.04    &     6.6e-2     &     1.09    &   2.5e-2    &     0.99   & 1.2e-2    &     1.00\\
7  &   3.2e-3    &    1.02    &     5.1e-1     &     -2.95   &  1.3e-2     &     0.98   & 6.2e-3    &     1.00\\
\bottomrule
\end{tabular}
\vskip3mm
\end{center}
\caption{ Example \ref{Ex1}; errors and convergence orders with $\rho^- = 1$ and $\rho^+ = 10^4$, using method \eqref{fem}-\eqref{a_h5} with stabilization parameters $\gamma = 10$ and $\gamma_F = 10$.}
\label{Table:Ex1_1}
\end{table}

\begin{table}[!htp]
\small
\ra{1.1}
\begin{center}
\begin{tabular}{@{}l@{\hskip .5in}c@{\hskip .2in}c@{\hskip .3in}c@{}}\toprule
$\rho^+$ & $e_h^{0}$  & $\bar{e}_h^{1,\infty}$& $e_h^{1}$     \\ \midrule
$10^1$    & 2.3e-6  & 6.5e-3 & 2.8e-3\\
$10^2$    & 2.0e-6  & 6.6e-3 & 2.0e-3\\
$10^3$    & 2.0e-6  & 6.6e-3 & 1.9e-3\\
$10^4$    & 2.0e-6  & 6.6e-3 & 1.9e-3\\
$10^5$    & 2.0e-6  & 6.6e-3 & 1.9e-3\\
$10^6$    & 2.0e-6  & 6.6e-3 & 1.9e-3\\
\bottomrule
\end{tabular}\vskip3mm
\end{center}
\caption{ Example \ref{Ex1}; errors with $\rho^- = 1$ and $h= 2^{-(6+3/2)}$, using method \eqref{fem}-\eqref{a_h5} with stabilization parameters $\gamma = 10$ and $\gamma_F = 10$.}
\label{Table:Ex1_2}
\end{table}

The results in Table \ref{Table:Ex1_1} show optimal convergence for the errors $e_h^1$ and $e_h^0$ validating the theoretical results Theorem \ref{maintheorem} and \ref{coro2}. In addition we observe optimal convergence for the error $e_h^\infty$ and $e_h^{1,\infty}$.
The error $\bar{e}_h^1$,  weighted with $\rho$ instead of $\sqrt{\rho}$,  converges optimally. However, the rate of convergence of the error {$\bar{e}_h^{1,\infty}$} is not optimal. An interesting observation is that the error {$\tilde{e}_h^{1,\infty}$} converges optimally, indicating that is only a couple of elements where the error does not converge.
Another appealing feature of the method is the optimal order of convergence for the error $e_h^{\n{n},\infty}$.

Table \ref{Table:Ex1_2} shows that the errors are independent of the contrast. We can observe that although we increase $\rho^+$ the errors remain constant, showing the contrast independency of our estimates.

We present in Table \ref{Table:Ex1_3} the errors and convergence orders for the method analyzed in the paper \eqref{fem}-\eqref{a_h}.

\begin{table}[!htp]
\small
\ra{1.1}
\begin{center}
\begin{tabular}{@{}l@{\hskip .4in}c@{\hskip .2in}c@{\hskip .4in}c@{\hskip .2in}c@{\hskip .4in}c@{\hskip .2in}c@{\hskip .4in}c@{\hskip .2in}c@{}}\toprule
$l$ & $e_h^0$ & \revblue{e.o.c.} & $e_h^\infty$ & \revblue{e.o.c.}  & $e_h^1$ & \revblue{e.o.c.} & $e_h^{1,\infty}$ & \revblue{e.o.c.} \\ \midrule
1  &   8.6e-3   &          &    2.6e-2    &           &    1.1e-1    &              &     3.7e-1    &            \\
2  &   1.7e-3   &   2.31   &    6.0e-3    &   2.14    &    4.5e-2    &     1.30     &     2.0e-1    &       0.90\\
3  &   2.8e-4   &   2.63   &    1.3e-3    &   2.18    &    1.8e-2    &     1.31     &     9.3e-2    &       1.09\\
4  &   4.7e-5   &   2.57   &    3.4e-4    &   1.97    &    8.4e-3    &     1.12     &     5.5e-2    &       0.76\\
5  &   9.2e-6   &   2.35   &    7.6e-5    &   2.16    &    4.0e-3    &     1.08     &     2.6e-2    &       1.09\\
6  &   2.0e-6   &   2.20   &    1.9e-5    &   2.01    &    1.9e-3    &     1.03     &     1.4e-2    &       0.91\\
7  &   4.7e-7   &   2.09   &    5.1e-6    &   1.88    &    9.5e-4    &     1.02     &     7.1e-2    &       -2.37\\
\bottomrule
\end{tabular}\vskip3mm
\begin{tabular}{@{}l @{\hskip .4in}c@{\hskip .2in}c@{\hskip .4in}c@{\hskip .2in}c@{\hskip .4in}c@{\hskip .2in}c@{\hskip .4in}c@{\hskip .2in}c@{}}\toprule
$l$ &  $\bar{e}_h^1$ & \revblue{e.o.c.} & $\bar{e}_h^{1,\infty}$ & \revblue{e.o.c.} & $\tilde{e}_h^{1,\infty}$ & \revblue{e.o.c.}& $e_h^{\n{n},\infty}$ & \revblue{e.o.c.}    \\ \midrule
1     &     4.0e-1   &            &      7.4e-1    &                 & 7.4e-1    &            & 3.7e-1    &        \\
2     &     1.6e-1   &     1.33   &      2.5e+0    &      -1.79      & 4.0e-1    &      0.90  & 2.0e-1    &     0.90\\
3     &     6.4e-2   &     1.31   &      2.8e-1    &      3.20       & 1.9e-1    &      1.07  & 9.2e-2    &     1.11\\
4     &     2.9e-2   &     1.15   &      1.6e+0    &      -2.56      & 9.7e-2    &      0.96  & 4.5e-2    &     1.01\\
5     &     1.4e-2   &     1.09   &      4.8e-1    &      1.76       & 4.7e-2    &      1.04  & 2.3e-2    &     0.96\\
6     &     6.6e-3   &     1.05   &      4.8e-1    &      0.01       & 2.4e-2    &      1.00  & 1.1e-2    &     1.05\\
7     &     3.2e-3   &     1.02   &      7.1e+0    &      -3.88      & 1.8e-2    &      0.42  & 5.9e-3    &     0.93\\
\bottomrule
\end{tabular}\vskip3mm
\end{center}
\caption{ Example \ref{Ex1}; errors and convergence orders with $\rho^- = 1$ and $\rho^+ = 10^4$, using method \eqref{fem}-\eqref{a_h} with stabilization parameters $\gamma = 10$.}
\label{Table:Ex1_3}
\end{table}

The results in Table \ref{Table:Ex1_3} show optimal convergence for the errors $e_h^1$ and $e_h^0$ validating the theoretical results Theorem \ref{maintheorem} and \ref{coro2}. In addition we observe optimal convergence for the error $e_h^\infty$. The error $e_h^{1,\infty}$ seems to converge optimal up to mesh $l=6$, and then for the last mesh the rate can possibly be affected by the choice of the flux stabilization parameter. As in the previous test the error $\bar{e}_h^1$ converges optimally, however for the rest of the errors the convergence is not as clear as in the previous method. We suspect that this  phenomena is related to the weights $|e^\pm|$ in the flux stabilization.

Finally, Table \ref{Table:Ex1_4} displays error and convergence orders for the method without flux stabilization \eqref{fem}-\eqref{a_h4}.
\begin{table}[!htp]
\small
\ra{1.1}
\begin{center}
\begin{tabular}{@{}l@{\hskip .4in}c@{\hskip .2in}c@{\hskip .4in}c@{\hskip .2in}c@{\hskip .4in}c@{\hskip .2in}c@{\hskip .4in}c@{\hskip .2in}c@{}}\toprule
$l$ & $e_h^0$ & \revblue{e.o.c.} & $e_h^\infty$ & \revblue{e.o.c.}  & $e_h^1$ & \revblue{e.o.c.} & $e_h^{1,\infty}$ & \revblue{e.o.c.} \\ \midrule
1  &   1.7e-3   &             &    7.0e-3    &              &      5.8e-2    &                &      2.2e-1     &\\
2  &   4.5e-3   &   1.88   &    2.9e-3    &   1.29    &      3.1e-2    &     0.89    &      1.6e-1     &      0.51\\
3  &   3.6e-4   &   0.31   &    6.1e-3    &   -1.08   &      1.2e-1    &     -1.89    &      1.3e+1     &      -6.37\\
4  &   2.8e-5   &   3.69   &    2.1e-4    &   4.87    &      7.7e-3    &     3.92    &      2.0e-1     &      5.98\\
5  &   7.3e-6   &   1.96   &    5.0e-5    &   2.05    &      3.8e-3    &     1.02    &      1.5e-1     &      0.40\\
6  &   1.7e-6   &   2.06   &    1.2e-5    &   2.05    &      1.9e-3    &     1.01    &      2.5e-2     &      -0.70\\
7  &   4.5e-7   &   1.95   &    4.0e-6    &   1.60    &      9.5e-4    &     0.99    &      1.1e-1     &      0.23 \\
\bottomrule
\end{tabular}\vskip3mm
\begin{tabular}{@{}l @{\hskip .4in}c@{\hskip .2in}c@{\hskip .4in}c@{\hskip .2in}c@{\hskip .4in}c@{\hskip .2in}c@{\hskip .4in}c@{\hskip .2in}c@{}}\toprule
$l$& $\bar{e}_h^1$ & \revblue{e.o.c.} & $\bar{e}_h^{1,\infty}$ & \revblue{e.o.c.} & $\tilde{e}_h^{1,\infty}$ & \revblue{e.o.c.}& $e_h^{\n{n},\infty}$ & \revblue{e.o.c.}    \\ \midrule
1  &  2.1e-1    &           &     5.8e-1     &             &  2.2e-1    &            & 1.5e-1    &         \\
2  &  1.3e-1    &  0.69     &     1.4e+1     &     -4.58   &  1.5e-1    &      0.49  & 1.1e-1    &     0.39\\
3  &  7.3e-1    & -2.46     &     1.3e+2     &     0.10    &  8.9e-1    &      -2.5  & 1.3e+1    &     -6.83\\
4  &  2.8e-2    &  4.70     &     7.2e+0     &     0.84    &  5.3e-2    &      4.07  & 2.0e-1    &     6.00\\
5  &  1.3e-2    &  1.07     &     2.6e+0     &     1.50    &  2.9e-2    &      0.86  & 1.5e-1    &     0.39\\
6  &  7.0e-3    &  0.93     &     2.1e+0     &     0.27    &  3.7e-2    &      -0.35 & 2.5e-1    &     -0.71\\
7  &  3.3e-3    &  0.99     &     8.6e+0     &     -2.10   &  2.1e-2    &      0.86  & 1.1e-3    &      0.23\\
\bottomrule
\end{tabular}\vskip3mm
\end{center}
\caption{ Example \ref{Ex1}; errors and convergence orders with $\rho^- = 1$ and $\rho^+ = 10^4$, using method \eqref{fem}-\eqref{a_h4} with stabilization parameters $\gamma = 10$.}
\label{Table:Ex1_4}
\end{table}

The results in Table \ref{Table:Ex1_4} show optimal asymptotical convergence for the errors $e_h^1$ and $e_h^0$. In addition we observe a slightly sub-optimal convergence for the error $e_h^\infty$, approximately $1.8$. The error $e_h^{1,\infty}$ does not seem to converge.
As in the previous test the error $\bar{e}_h^1$ converges asymptotically to 1, however for the rest of the errors we do not observe convergence.

\begin{figure}
  \centering
  \includegraphics[scale=.5]{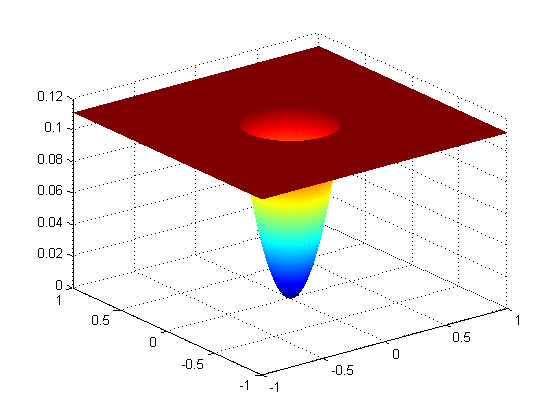}\includegraphics[scale=.5]{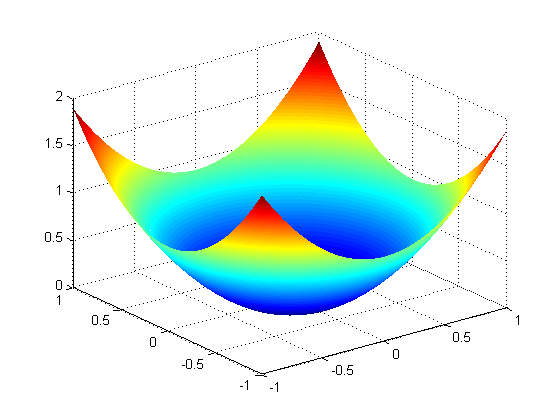}\\
  \caption{Approximate solution Example \ref{Ex1} (left) case  $\Gamma = \partial \Omega^-$ and for  case $\Gamma = \partial \Omega^+$ (right).}\label{Appsol}
\end{figure}
\end{enumerate}

\clearpage
\counterwithin{lemma}{section}
\appendix
\section{Technical lemmas}\label{technicallemmasection}
{
In this section we prove two technical lemmas involving geometrical estimates on elements intersected by the interface $\Gamma$. We remind the reader that we assume that $\Gamma$ is a simple $\mathcal{C}^2$ curve with an arc-length parameterization $\n{X}:[0,|\Gamma|)\rightarrow \Gamma$, and assumption (1)-(3) in Section \ref{SpaceVh}.}

{We assume that $r$ is the radius of our tubular neighborhood given in Lemma \ref{tubularlemma}. Equivalently,  for any  $0< {\tau}\le r$ we have
\begin{equation*}
Tub({\tau})=\{x: \text{dist}(x, \Gamma) \le {\tau}\}.
\end{equation*}
Moreover, for any $x \in Tub(r)$ there exists a unique $x_\Gamma \in \Gamma$ such that $\text{dist}(x,x_\Gamma)=\text{dist}(x, \Gamma)$ and $x-x_\Gamma$ is perpendicular to $\Gamma$ at $x_\Gamma$.}

{
\begin{lemma}\label{TGammalessLT}
Consider the $\Gamma_{s}$ a segment of the curve $\Gamma$ and $\ell$ the straight segment connecting the two end points of $\Gamma_{s}$. Assume that $|\ell| \le r$, where $r$ is the radius of the $r$-tubular neighborhood. Then, it holds
\begin{equation*}
|\Gamma_{s}| \leq 2 |\ell|.
\end{equation*}
\end{lemma}}

\begin{proof}
{
Let $x_1, x_2$ be the endpoints of the line segment $\ell$. For $i=1,2$, let $N_{i}$ be the line  that is normal to $\Gamma$  at $x_i$. Let $\tilde{S}$ be the infinite region enclosed by $N_{1}$ and $N_{2}$, then let $S= Tub(r/2) \cap \tilde{S}$ be the tubular section of $\Gamma_s$.  We see that the area of $S$ is given by $|S|= r|\Gamma_{s}|$.  Let $M_{i}$ $i=1,2$ be the two lines that are parallel to $\ell$ and distance $r$ from  $\ell$. Consider the trapezoid, $R$, enclosed by $N_{i}, M_{i}, i=1,2$.  We note that the area of $R$ is given by $|R|=2 r |\ell|$. The result will follow if we show that $S \subset R$ since this will imply that $r|\Gamma_{s}| \le 2r|\ell|$.}

{
To this end, we first note that the line segment $\ell \subset  Tub(r/2)$ since the distance of any point in $\ell$ to $\Gamma$ is less than $|\ell/2|$ which we are assuming is less than $r/2$.
Let $x \in S$, then we know there exists a unique point $x_\Gamma \in \Gamma_s$ where the distance from $x$ to $x_\Gamma$ is less than $r/2$ and the line, which we denote by $L_x$, that passes through $x$ and $x_\Gamma$ is perpendicular to $\Gamma$ at $x_\Gamma$. We know that $L_x$ intersects $\ell$ and we call this point $y$.  We also know that distance between $y$ and $x_\Gamma$ is less than $r/2$ since $y \in \ell \subset Tub(r/2)$. Hence, we have shown that $\text{dist}(y,x) \le r$. This, of course implies that the distance of $x$ to the infinite line $\ell_{\text{ext}}$  (which is the line that contains the line segment $\ell$) is  at most $r$. In other words, $x$ is in between the two lines $M_1$ and $M_2$. Since $x$ was in the tubular section $S$ it was in between the two lines $N_1$ and $N_2$ and hence $x$ belongs to the trapezoid $R$.
}
\end{proof}

\begin{SCfigure}[1.8][!htbp]
\caption{Illustration of a curve segment $\Gamma_{s}$, straight line segment $\ell$ connecting the end points of $\Gamma_{s}$, the tubular section of radius $S$, and the trapezoid $R$.}
\includegraphics[scale=.3]{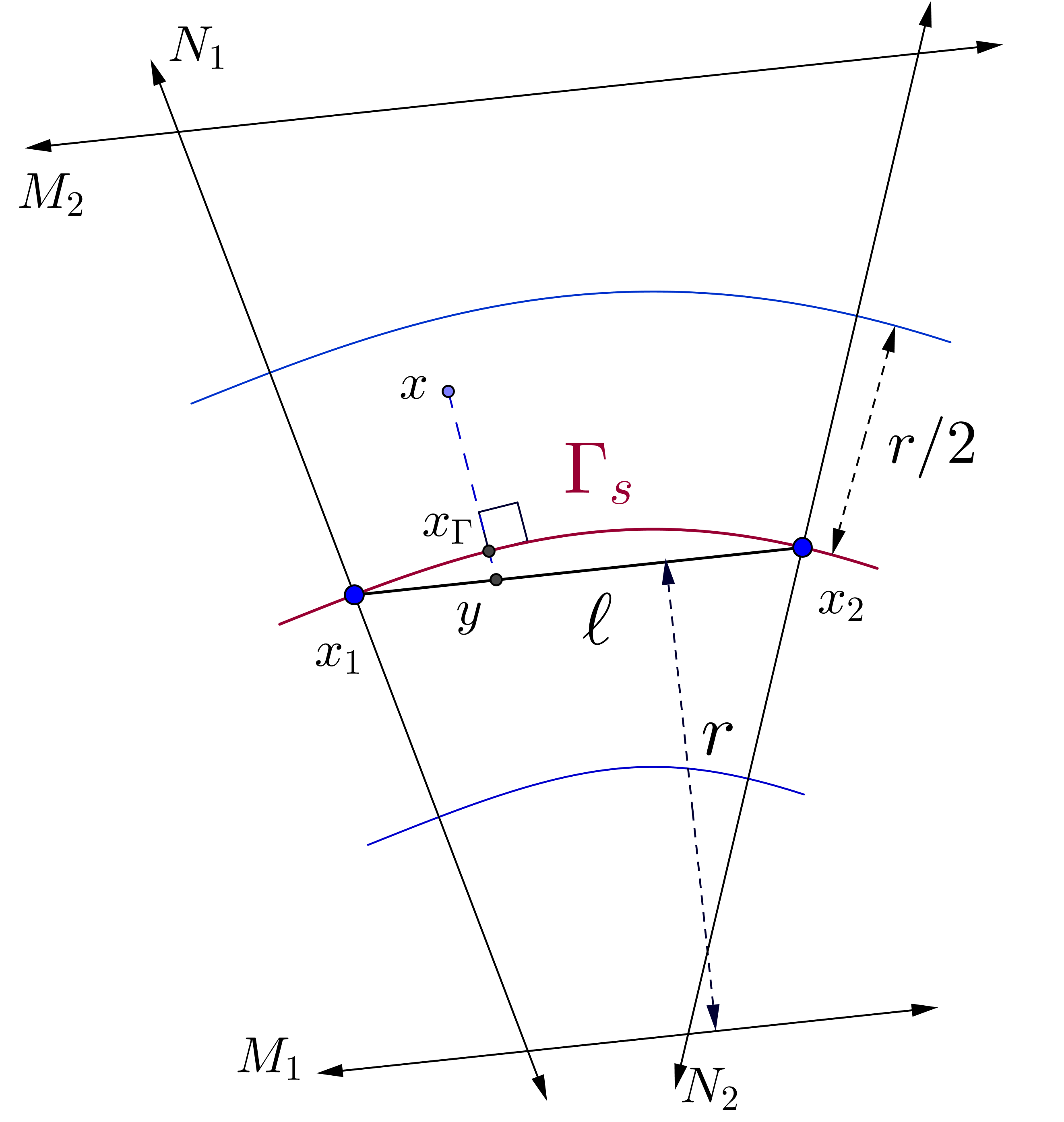}
\label{figuretubular}
\end{SCfigure}
{
\subsection{Proof of Lemma \ref{twotri}.}\label{twotriproof}
Let $e=\mathrm{Int}(\partial T_1\cap \partial T_2)\in \EhG$, with $T_1, T_2\in \Th$, and the previous definitions of $e^\pm = e\cap \Omega^\pm$. We analyze the case in $\Omega^-$. The same analysis is valid in $\Omega^+$. We proceed analyzing two cases depending on the intersection of the edge $e$ with $\Gamma$:
\begin{enumerate}[label=(\roman*)]
\item If $ e\cap\Gamma = \emptyset $, ($e=e^-$). If either $T_1$ or $T_2$ does not belong to $\ThG$ the results is trivial. Consider the case where $T_1$ and $T_2$ belong to $\ThG$, the interface sections $T_{1,\Gamma}$ and $T_{2,\Gamma}$ are nonempty. We consider the midpoint $m_e$ the segment $e$ and the ball $B_{|e|/2}(m_e)$ of radius $|e|/ 2$ and centered at $m_e$. If the interface does not cross the ball  $B_{|e|/ 2}(m_e)$ then we have
\[
|T_1^{-}|\geq  |B_{|e|/ 2}(m_e) \cap T_1|.
\]
Denote by $\underline{\alpha}$  the minimum angle of the triangulation (given by shape regularity), and let $\tilde{\alpha}=\min\{\underline{\alpha}, \pi/4\}$. Now consider the isosceles triangle $\tilde{T}_1$ with base edge $e$ and base angles $\tilde{\alpha}$. Then, clearly $\tilde{T}_1 \subset B_{|e|/ 2}(m_e) \cap T_1$ and $|\tilde{T}_1 |=(|e^-|/2)^2\tan\{\tilde{\alpha}\}$. Therefore,  $|T_1^{-}| \ge (|e^-|/2)^2\tan\{\tilde{\alpha}\}$. Assume then that the interface crosses the ball $B_{|e|/2}(m_e)$ in $T_1$.  Observe that this implies that there exists a point in $T_{1,\Gamma}$ who's normal passes through $m_e$ with  distance less than $|e|/2$. Now, if $T_{2,\Gamma}$ crosses the ball then we will have two points on $\Gamma$ at a distance less than $|e|/2$ whose normal passes through $m_e$ which contradicts the tubular neighborhood assumption. Therefore
\[
\max_{i=\{1,2\}}|T_{i}^{-}|\geq  \Big(\frac{|e|}{2}\Big)^2\tan\{\tilde{\alpha}\}.
\]
\item If $e\cap\Gamma \neq \emptyset$.  Similarly as in the previous case, consider $m_e$ the midpoint of $e^{-}$ and the ball of radius $|e^{-}|/2$ centered at $m_e$. We observe that this ball can not cross both $T_{1,\Gamma}$ and $T_{2,\Gamma}$. Therefore
\[
\max_{i=\{1,2\}}|T_{i}^{-}|\geq  \Big(\frac{|e^-|}{2}\Big)^2\tan\{\tilde{\alpha}\}.
\]
\end{enumerate}
}
{
\subsection{ Proof  of Lemma \ref{compare}.}\label{compareproof}
We analyze the case in $\Omega^-$. We first observe that, by triangle inequality we have
\begin{equation*}
|l_T|\,\,\leq\sum_{i=1}^3 |e^{-}_i|\,\,\leq\,\,3\,\max_{i=\{1,2,3\}} |e^-|.
\end{equation*}
Therefore, using Lemma \ref{TGammalessLT}
 \begin{equation*}
|T_\Gamma|\,\,\leq\,\, 6\max_{i=1,2,3} |e_i^-|.
\end{equation*}
Same analysis is valid to prove the statement in $\Omega^+$.
}

\section{Construction of basis functions of $S^{1}(T)$}\label{AppendixS1T}
{
In this section we define basis functions $\{w_1,w_2,w_3\}$ of the space $S^{1}(T)$, for $T\in\mathcal{T}_h^{\Gamma}$. Consider $\{x_1,x_2,x_3\}$ nodes of the element $T$. Let $\{\lambda_1,\lambda_2,\lambda_3\}$ be the barycentric coordinates of a point $x\in T$ with respect to $\{x_1,x_2,x_3\}$. Consider the following representation
\begin{equation*}
w_i(x) = \left\{
        \begin{array}{ll}
          w_i^-(x), & \hbox{if }x\in T^-; \\
          w_i^+(x), & \hbox{if }x\in T^+.
        \end{array}
      \right. w_i^{\pm}(x) = \sum_{j = 1}^{3} a_{i,j}^{\pm} \lambda_{j}(x).
\end{equation*}
We construct $\{w_1,w_2,w_3\}$, a set of basis functions of the space $S^1(T)$, by satisfying the following conditions: for $i =1,2,3$
\begin{equation*}
\left\{
\begin{array}{ll}
 w_{i}(x_j) = \delta_{i,j}&  \hbox{ for } j=1,2,3; \\
 {[ w_i(x_0) ]}= 0 & x_0\in \TG;\\
 {[D_{\n{t}_0}w_i]} = 0& \n{t}_0 = \n{t}(x_0);\\
 {[D_{\n{n}_0}w_i = 0]}& \n{n}_0 = \n{n}(x_0);
\end{array}
\right.\quad  \hbox{where}\quad  \delta_{i,j} = \left\{
                           \begin{array}{ll}
                             1, & \hbox{if } i=j \\
                             0, & \hbox{if } i\neq j.
                           \end{array}
                         \right.
\end{equation*}
These conditions are written in a system of size $6\times 6$ using the barycentric coordinates representation of $w_i^{\pm}$, i.e., we find the unknowns coefficients $\{a_{i,j}^{\pm}\}_{j=1}^{3}$ of $w_i$ for $i=1,2,3$, solutions of:
{\small
\begin{equation*}
\left(
  \begin{array}{cccccc}
    \tilde{\delta}_{1,+} & 0              & 0              & \tilde{\delta}_{1,-}   & 0              & 0               \\
    0              & \tilde{\delta}_{2,+} & 0              & 0              & \tilde{\delta}_{2,-}   & 0               \\
    0              & 0              & \tilde{\delta}_{3,+} & 0              & 0              & \tilde{\delta}_{3,-}    \\
    \lambda_1(x_0) & \lambda_2(x_0) & \lambda_3(x_0) &     -\lambda_1(x_0) & -\lambda_2(x_0) & -\lambda_3(x_0) \\
    -D_{\n{t}_0} \lambda_1 & -D_{\n{t}_0} \lambda_2 &  -D_{\n{t}_0} \lambda_3& D_{\n{t}_0} \lambda_1 & D_{\n{t}_0} \lambda_2 &D_{\n{t}_0} \lambda_3  \\
-\rho^{+} D_{\n{n}_0} \lambda_1 & -\rho^{+}D_{\n{n}_0} \lambda_2 &  \rho^{+}D_{\n{n}_0} \lambda_3& \rho^{-}D_{\n{n}_0} \lambda_1 & \rho^{-}D_{\n{n}_0} \lambda_2 &\rho^{-}D_{\n{n}_0} \lambda_3  \\
  \end{array}
\right)\left(
         \begin{array}{c}
           a_{i,1}^+ \\
           a_{i,2}^+ \\
           a_{i,3}^+ \\
           a_{i,1}^- \\
           a_{i,2}^- \\
           a_{i,3}^- \\
         \end{array}
       \right) = \left(
         \begin{array}{c}
           \delta_{i,1} \\
           \delta_{i,2} \\
           \delta_{i,3} \\
           0 \\
           0 \\
           0\\
         \end{array}
       \right)
\end{equation*}
}
where
\begin{equation*}
\tilde{\delta}_{j,\pm} = \left\{
                           \begin{array}{ll}
                             1, & \hbox{if } x_{j}\in T^{\pm} \\
                             0, & \hbox{if } x_{j}\in T^{\mp}.
                           \end{array}
                         \right.
\end{equation*}
}

\bibliography{paper}{}

\begin{thebibliography}{10}

\bibitem{MR3218337}
Slimane Adjerid, Mohamed Ben-Romdhane, and Tao Lin.
\newblock Higher degree immersed finite element methods for second-order
  elliptic interface problems.
\newblock {\em Int. J. Numer. Anal. Model.}, 11(3):541--566, 2014.

\bibitem{MR2981355}
Nelly Barrau, Roland Becker, Eric Dubach, and Robert Luce.
\newblock A robust variant of {NXFEM} for the interface problem.
\newblock {\em C. R. Math. Acad. Sci. Paris}, 350(15-16):789--792, 2012.

\bibitem{MR2571349}
Roland Becker, Erik Burman, and Peter Hansbo.
\newblock A {N}itsche extended finite element method for incompressible
  elasticity with discontinuous modulus of elasticity.
\newblock {\em Comput. Methods Appl. Mech. Engrg.}, 198(41-44):3352--3360,
  2009.

\bibitem{MR1278258}
Susanne~C. Brenner and L.~Ridgway Scott.
\newblock {\em The mathematical theory of finite element methods}, volume~15 of
  {\em Texts in Applied Mathematics}.
\newblock Springer-Verlag, New York, 1994.

\bibitem{MR2738930}
Erik Burman.
\newblock Ghost penalty.
\newblock {\em C. R. Math. Acad. Sci. Paris}, 348(21-22):1217--1220, 2010.

\bibitem{MR3264337}
Erik Burman and Peter Hansbo.
\newblock Fictitious domain methods using cut elements: {III}. {A} stabilized
  {N}itsche method for {S}tokes' problem.
\newblock {\em ESAIM Math. Model. Numer. Anal.}, 48(3):859--874, 2014.

\bibitem{MR3051411}
Erik Burman and Paolo Zunino.
\newblock Numerical approximation of large contrast problems with the unfitted
  {N}itsche method.
\newblock In {\em Frontiers in numerical analysis---{D}urham 2010}, volume~85
  of {\em Lect. Notes Comput. Sci. Eng.}, pages 227--282. Springer, Heidelberg,
  2012.

\bibitem{MR2684351}
C.-C. Chu, I.~G. Graham, and T.-Y. Hou.
\newblock A new multiscale finite element method for high-contrast elliptic
  interface problems.
\newblock {\em Math. Comp.}, 79(272):1915--1955, 2010.

\bibitem{MR0394451}
Manfredo~P. do~Carmo.
\newblock {\em Differential geometry of curves and surfaces}.
\newblock Prentice-Hall, Inc., Englewood Cliffs, N.J., 1976.
\newblock Translated from the Portuguese.

\bibitem{MR1273155}
Maksymilian Dryja and Olof~B. Widlund.
\newblock Domain decomposition algorithms with small overlap.
\newblock {\em SIAM J. Sci. Comput.}, 15(3):604--620, 1994.
\newblock Iterative methods in numerical linear algebra (Copper Mountain
  Resort, CO, 1992).

\bibitem{MR737190}
David Gilbarg and Neil~S. Trudinger.
\newblock {\em Elliptic partial differential equations of second order}, volume
  224 of {\em Grundlehren der Mathematischen Wissenschaften [Fundamental
  Principles of Mathematical Sciences]}.
\newblock Springer-Verlag, Berlin, second edition, 1983.

\bibitem{MR2377272}
Yan Gong, Bo~Li, and Zhilin Li.
\newblock Immersed-interface finite-element methods for elliptic interface
  problems with nonhomogeneous jump conditions.
\newblock {\em SIAM J. Numer. Anal.}, 46(1):472--495, 2007/08.

\bibitem{MR2681555}
Yan Gong and Zhilin Li.
\newblock Immersed interface finite element methods for elasticity interface
  problems with non-homogeneous jump conditions.
\newblock {\em Numer. Math. Theory Methods Appl.}, 3(1):23--39, 2010.

\bibitem{MR775683}
P.~Grisvard.
\newblock {\em Elliptic problems in nonsmooth domains}, volume~24 of {\em
  Monographs and Studies in Mathematics}.
\newblock Pitman (Advanced Publishing Program), Boston, MA, 1985.

\bibitem{MR1941489}
Anita Hansbo and Peter Hansbo.
\newblock An unfitted finite element method, based on {N}itsche's method, for
  elliptic interface problems.
\newblock {\em Comput. Methods Appl. Mech. Engrg.}, 191(47-48):5537--5552,
  2002.

\bibitem{MR2192480}
Peter Hansbo.
\newblock Nitsche's method for interface problems in computational mechanics.
\newblock {\em GAMM-Mitt.}, 28(2):183--206, 2005.

\bibitem{MR2740492}
Xiaoming He, Tao Lin, and Yanping Lin.
\newblock Immersed finite element methods for elliptic interface problems with
  non-homogeneous jump conditions.
\newblock {\em Int. J. Numer. Anal. Model.}, 8(2):284--301, 2011.

\bibitem{MR2864671}
Xiaoming He, Tao Lin, and Yanping Lin.
\newblock The convergence of the bilinear and linear immersed finite element
  solutions to interface problems.
\newblock {\em Numer. Methods Partial Differential Equations}, 28(1):312--330,
  2012.

\bibitem{MR1929889}
Jianguo Huang and Jun Zou.
\newblock Some new a priori estimates for second-order elliptic and parabolic
  interface problems.
\newblock {\em J. Differential Equations}, 184(2):570--586, 2002.

\bibitem{MR2018791}
Zhilin Li, Tao Lin, and Xiaohui Wu.
\newblock New {C}artesian grid methods for interface problems using the finite
  element formulation.
\newblock {\em Numer. Math.}, 96(1):61--98, 2003.

\bibitem{MR3338673}
Tao Lin, Yanping Lin, and Xu~Zhang.
\newblock Partially penalized immersed finite element methods for elliptic
  interface problems.
\newblock {\em SIAM J. Numer. Anal.}, 53(2):1121--1144, 2015.

\bibitem{MR3268662}
Andr{{\'e}} Massing, Mats~G. Larson, Anders Logg, and Marie~E. Rognes.
\newblock A stabilized {N}itsche fictitious domain method for the {S}tokes
  problem.
\newblock {\em J. Sci. Comput.}, 61(3):604--628, 2014.

\bibitem{MR3248049}
Andr{{\'e}} Massing, Mats~G. Larson, Anders Logg, and Marie~E. Rognes.
\newblock A stabilized {N}itsche overlapping mesh method for the {S}tokes
  problem.
\newblock {\em Numer. Math.}, 128(1):73--101, 2014.

\bibitem{MR3047906}
Paolo Zunino.
\newblock Analysis of backward {E}uler/extended finite element discretization
  of parabolic problems with moving interfaces.
\newblock {\em Comput. Methods Appl. Mech. Engrg.}, 258:152--165, 2013.

\bibitem{MR2820966}
Paolo Zunino, Laura Cattaneo, and Claudia~Maria Colciago.
\newblock An unfitted interface penalty method for the numerical approximation
  of contrast problems.
\newblock {\em Appl. Numer. Math.}, 61(10):1059--1076, 2011.

\end{thebibliography}
\bibliographystyle{plain}

\vskip20mm
\end{document}